\documentclass[12pt]{amsart}
\usepackage{amsmath,amsthm,amsfonts,amscd,amssymb,eucal,latexsym,mathrsfs, appendix}
\usepackage[numbers,sort&compress]{natbib}
\newtheorem{theorem}{Theorem}[section]
\newtheorem{corollary}[theorem]{Corollary}
\newtheorem{lemma}[theorem]{Lemma}
\newtheorem{proposition}[theorem]{Proposition}

\theoremstyle{definition}
\newtheorem{definition}[theorem]{Definition}
\newtheorem{remark}[theorem]{Remark}
\newtheorem{notation}[theorem]{Notation}
\newtheorem{example}[theorem]{Example}

\newcommand{\htopol}{{\text{\rm h}}_{\text{\rm top}}}
\newcommand{\rank}{{\rm rank}}

\newcommand{\diag}{{\rm diag}}

\newcommand{\cB}{{\mathcal B}}

\newcommand{\cL}{{\mathcal L}}

\newcommand{\cD}{{\mathcal D}}

\newcommand{\cU}{{\mathcal U}}
\newcommand{\cM}{{\mathcal M}}

\newcommand{\cQ}{{\mathcal Q}}
\newcommand{\cP}{{\mathcal P}}

\newcommand{\cV}{{\mathcal V}}

\newcommand{\Cb}{{\mathbb C}}
\newcommand{\Eb}{{\mathbb E}}
\newcommand{\Mb}{{\mathbb M}}
\newcommand{\Zb}{{\mathbb Z}}
\newcommand{\Tb}{{\mathbb T}}
\newcommand{\Qb}{{\mathbb Q}}
\newcommand{\Rb}{{\mathbb R}}
\newcommand{\Nb}{{\mathbb N}}
\newcommand{\sB}{{\mathscr B}}
\newcommand{\sA}{{\mathscr A}}

\newcommand{\tr}{{\rm tr}}

\newcommand{\GL}{{\rm GL}}

\newcommand{\M}{M}

\newcommand{\rh}{{\rm h}}
\newcommand{\rH}{{\rm H}}

\newcommand{\ddet}{{\rm det}}

\allowdisplaybreaks

\begin{document}

\title[Entropy and Determinant]{Compact Group Automorphisms, Addition Formulas and Fuglede-Kadison Determinants}
\author{Hanfeng Li}
\thanks{Partially supported by NSF Grants DMS-0701414 and DMS-1001625.}
\address{\hskip-\parindent
Department of Mathematics, Chongqing University, 
Chongqing 401331, China.
\break
Department of Mathematics, SUNY at Buffalo,
Buffalo, NY 14260-2900, U.S.A.}
\email{hfli@math.buffalo.edu}


\begin{abstract}
For a countable amenable group $\Gamma$ and an element $f$ in the integral group ring $\Zb\Gamma$ being invertible in the group von Neumann algebra of $\Gamma$, we show that the entropy of the shift action of $\Gamma$ on the Pontryagin dual of the quotient of $\Zb\Gamma$ by its left ideal generated by $f$ is the logarithm of the Fuglede-Kadison determinant of $f$. For the proof, we establish an $\ell^p$-version of Rufus Bowen's definition of topological entropy, addition formulas for group extensions of countable amenable group actions, and an approximation formula for the Fuglede-Kadison determinant of $f$ in terms of the determinants of perturbations of the compressions of $f$.
\end{abstract}

\date{December 28, 2011}

\maketitle

\section{Introduction} \label{introduction:sec}

There are two motivations for this paper.
First, for topological or measure-preserving actions of countable amenable groups, one has the entropy defined \cite{JMO, OW}.
But unlike the case of $\Zb$-actions or $\Zb^d$-actions (for $2\le d<\infty$), not many examples have been calculated
for nonabelian group actions.
Second, the study of automorphisms of compact metrizable groups has drawn much attention in the development of ergodic theory, because of the rich interplay
between dynamics and compact group structures. Though the $\Zb$-actions of compact metrizable groups by automorphisms are well understood (cf. \cite{Lind, LW88, MT, Yuz2, Yuz1}),
and much is known for $\Zb^d$-actions (cf. \cite{LSW, SW, Ward92,  Sch, RS, LS99, LS02, Schmidt01,  EW,  KS00}), very little has been understood for general countable amenable group actions (cf. \cite{BG, Den, DSch, ER, MB, MW}).

In this paper, we calculate the entropy for a rich class of
actions of countable amenable groups on compact metrizable groups by automorphisms, providing some steps towards
understanding the entropy theory of  such algebraic actions.

Let $\Gamma$ be a countable amenable group, and let $f$ be an element in the integral group ring $\Zb\Gamma$.
One may consider the quotient
$\Zb\Gamma/\Zb\Gamma f$ of $\Zb\Gamma$ by the left ideal $\Zb\Gamma f$ generated by $f$.
Then $\Gamma$ acts on the abelian group $\Zb\Gamma/\Zb\Gamma f$
by automorphisms via left translation, and hence
acts on its Pontryagin dual (a compact metrizable abelian group)
$$ X_f:=\widehat{\Zb\Gamma/\Zb\Gamma f}$$
by automorphisms. Denote the latter action by $\alpha_f$.
Explicitly, $X_f$ consists of elements $h$
in $(\Rb/\Zb)^{\Gamma}$ satisfying
$$\sum_{\gamma \in \Gamma}f_{\gamma}h_{\gamma^{-1}\gamma'}=0$$
for
all $\gamma'\in \Gamma$, and the action $\alpha_f$ is the restriction of
the right shift action of $\Gamma$ on $(\Rb/\Zb)^{\Gamma}$ to $X_f$, i.e.,
$(\gamma h)_{\gamma'}=h_{\gamma'\gamma}$ for all $h\in X_f$ and $\gamma, \gamma'\in \Gamma$ (see Section~\ref{S-finite group}).
The topological entropy and the measure-theoretical entropy (with respect to the normalized Haar measure) of $\alpha_f$ coincide
\cite{Den}, and will be denoted by $\rh(\alpha_f)$.

When $\Gamma=\Zb$, one may identify $\Zb\Gamma$ with the one-variable Laurent polynomial ring $\Zb[u^{\pm 1}]$ via identifying
$1\in \Zb=\Gamma$ with $u$. Writing $f\in \Zb\Gamma$ as $u^{-k}(\sum_{j=0}^nc_ju^j)$ with $n\ge 0$ and $c_nc_0\neq 0$, and denoting by $\lambda_1, \dots, \lambda_n$ the roots of $\sum_{j=0}^nc_ju^j$,
Yuzvinski\u{\i} \cite{Yuz2}
showed that
\begin{eqnarray} \label{E-Z}
\rh(\alpha_f)=\log |c_n|+\sum_{j=1}^n\log^+|\lambda_j|,
\end{eqnarray}
where $\log^+t=\log \max(1, t)$ for $t\ge 0$. In general, Yuzvinski\u{i} calculated the entropy of any endomorphism of a compact metrizable group \cite{Yuz2}.

When $\Gamma=\Zb^d$ for some $1\le d<\infty$, one may identify $\Zb\Gamma$ with the $d$-variable Laurent polynomial ring $\Zb[u_1^{\pm 1}, \dots, u_d^{\pm 1}]$ naturally. For nonzero $f\in \Zb \Gamma=\Zb[u_1^{\pm 1}, \dots, u_d^{\pm 1}]$, Lind, Schmidt and Ward \cite{LSW, Sch} showed that
\begin{eqnarray} \label{E-Zd}
\rh(\alpha_f)=\log \Mb (f),
\end{eqnarray}
where $\Mb (f)$ is the {\it Mahler measure} of $f$ \cite{Mahler1, Mahler2} defined as
$$ \Mb(f)=\exp (\int_{\Tb^d}\log |f(s)|\, ds)$$
for $\Tb$ being the unit circle in $\Cb$ and $\Tb^d$ being endowed with the normalized Haar measure. (When $f=0$, clearly $\rh(\alpha_f)=\infty$.)
And this is the main step in their calculation for
the entropy of any action of $\Zb^d$ on a compact
metrizable group by automorphisms \cite{LSW, Sch}. In the case $d=1$, the calculation (\ref{E-Zd}) reduces to (\ref{E-Z}) via Jensen's formula.

Several years before Mahler introduced the Mahler measure, in \cite{FK} Fuglede and Kadison introduced a determinant $\ddet_{A}f$ for
invertible elements $f$ in a unital $C^*$-algebra $A$ with respect
to a tracial state $\tr_A$. It has found wide application in the study of $L^2$-invariants \cite{Luck}.
For a discrete group $\Gamma$, the group ring $\Zb\Gamma$ sits naturally in the left group von Neumann algebra
$\cL\Gamma$. Furthermore, $\cL\Gamma$ has a canonical tracial state $\tr_{\cL\Gamma}$. Thus one may consider
$\ddet_{\cL\Gamma} f$ for invertible $f\in \cL\Gamma$. When $\Gamma=\Zb^d$,  $f\in \Zb\Gamma$ is invertible in $\cL\Gamma$
if and only if $f$ has no zero point on $\Tb^d$. In such case, $\ddet_{\cL\Gamma} f$ is exactly
$\Mb(f)$ \cite{Den}.

In \cite{Den} Deninger pointed out the possibility of $\rh(\alpha_f)=\log \ddet_{\cL\Gamma}f$ for general countable amenable groups $\Gamma$
and $f\in \Zb\Gamma$, and confirmed  it in the special case
that $f$ is invertible in $\ell^1(\Gamma)$ (this is stronger than the condition that $f$ is invertible in
$\cL\Gamma$, see Appendix~\ref{A-invertibility}) and
positive in $\cL\Gamma$ and that $\Gamma$ has a $\log$-strong
F{\o}lner sequence. Deninger and Schmidt \cite{DSch} also confirmed it
in the special case that $f$ is invertible in
$\ell^1(\Gamma)$ and that $\Gamma$ is (amenable and) residually finite.
The connection between entropy, Mahler measure and Fuglede-Kadison determinant has been further explored by Deninger
in \cite{Den06, Den07, Den09}.

Our main result in this paper is

\begin{theorem} \label{T-main}
Let $\Gamma$ be a countable amenable group and let $f\in \Zb\Gamma$ be invertible in $\cL\Gamma$. Then
$$ \rh(\alpha_f)=\log \ddet_{\cL\Gamma} f.$$
\end{theorem}

One of the dynamical consequences of Theorem~\ref{T-main} and the general properties of the Fuglede-Kadison determinant
is that under the hypothesis of Theorem~\ref{T-main}, the actions $\alpha_f$ and $\alpha_{f^*}$ have the same entropy, where $f^*$ is the adjoint of $f$ defined
as $(f^*)_{\gamma}=f_{\gamma^{-1}}$ for all $\gamma \in \Gamma$.
This is a very non-trivial fact, as a priori there is no relation between $\alpha_f$ and $\alpha_{f^*}$ unless $f$ is in the center of $\Zb\Gamma$.

Our proof of Theorem~\ref{T-main} consists of three steps.

In the first step, we establish Theorem~\ref{T-main} under the further assumption that $f$ is positive in $\cL\Gamma$.
Since the invertibility of $f$ in $\cL\Gamma$ means that $f^{-1}$ exists as a bounded linear operator on $\ell^2(\Gamma)$, while
Rufus Bowen's definition of topological entropy is taking the maximum of distances between finite orbits of points and should be thought of
an $\ell^{\infty}$-distance, we develop an $\ell^2$-version of his definition in Section~\ref{S-Bowen}, which is of independent interest.
Then we prove the positive case of Theorem~\ref{T-main} in Section~\ref{S-positive}, using an estimate of number of integral points in
balls and an approximation formula of Deninger for $\ddet_{\cL\Gamma}f$ in such case.

In the second step, we prove the Yuzvinski\u{\i} addition formula in Section~\ref{S-addition}, which says
that the entropy of a $\Gamma$-action on a compact metrizable group by automorphisms is the sum of the entropy of
the restriction of the action to an invariant closed subgroup and the entropy of the induced action on the quotient group.
This formula allows us to reduce the calculation for the entropy of one action to that for the entropy of simpler actions.
In fact, we establish addition formulas for group extensions in both topological and measure-theoretical settings, and
the formula in either of these settings implies the Yuzvinski\u{\i} addition formula. The proof for each of these addition formulas
employs  both topological and measure-theoretical tools, using generalization of the various fibre and conditional entropies studied
in \cite{DS}, and the addition formula for $\Gamma$-extensions in \cite{WZ, Dan} which in turn depends on Rudolph and Weiss's orbit equivalence method in \cite{RW}.

The third step is to prove $\rh(\alpha_f)\ge \log\ddet_{\cL\Gamma}f$ under the hypothesis in Theorem~\ref{T-main}.
Compared to the positive case in step one, the main difficulty here is that the compression of $f$ to a nonempty finite subset of $\Gamma$ map fail to
be invertible. Our method of dealing with this difficulty is to {\it perturb} the compression of $f$ to an invertible linear operator.
For this purpose, we establish an approximation formula for $\log \ddet_{\cL\Gamma}f$ in terms of the determinants of the compressions, in Section~\ref{S-determinant}. This uses an approximation formula for traces, initiated by  L\"{u}ck in work on $L^2$-invariants \cite{Luck94}
and extended by Schick in \cite{Schick01}. We complete the third step in Section~\ref{S-proof of ineqaulity}, using Ornstein and Weiss's theory of quasitiling in \cite{OW}.

The proof of Theorem~\ref{T-main}, which uses the fact that the Fuglede-Kadison determinants of $f$ and $f^*$ are equal,
is finished in Section~\ref{S-consequence}. Some dynamical consequences of the theorem including
the equality of $\rh(\alpha_f)$ and $\rh(\alpha_{f^*})$ are also established there. We recall some background in Section~\ref{S-preliminary}, and give a proof of the case $\Gamma$ is finite in Section~\ref{S-finite group}, which shows clearly how the entropy and the Fuglede-Kadison determinant
are connected via several equalities. In an appendix, we compare invertibility in $\ell^1(\Gamma)$ and $\cL\Gamma$.

Recently, entropy has been defined for continuous actions of a countable sofic group on compact metrizable spaces and measure-preserving
actions of a countable sofic group on standard probability measure spaces, with respect to a sofic approximation sequence of the sofic group \cite{Bow2, KL11a}.
The class of sofic groups include all discrete amenable groups and residually finite groups.
The sofic entropies coincide with the classical entropies when the sofic group is amenable \cite{Bow5, KL11b}.
For a countable residually finite (not necessarily amenable) group $\Gamma$ and an $f\in \Zb\Gamma$, when the sofic approximation
sequence of $\Gamma$ comes from a sequence of finite-index normal subgroups of $\Gamma$, in various cases it has been shown that
the sofic topological entropy
and the sofic measure entropy (for the normalized Haar measure of $X_f$) of $\alpha_f$ are equal to $\log \ddet_{\cL\Gamma} f$ \cite{Bow4, BowenLi, KL11a}.

Throughout this paper, for a group $G$, we denote by $e_G$ the identity element of $G$. For a discrete group
$\Gamma$, we write $\Cb[[\Gamma]]$, $\Rb[[\Gamma]]$ and $\Zb[[\Gamma]]$  for $\Cb^{\Gamma}$, $\Rb^{\Gamma}$ and $\Zb^{\Gamma}$ respectively.
 For a finite set $F$, we write $\Cb[F]$ for $\Cb^{F}$, and equip it
with the standard $\ell^2$-norm. For a Hilbert space $H$, we denote by $B(H)$ the set of bounded linear operators on $H$, and equip it with the operator norm $\|\cdot \|$.

After this paper was finished, Douglas Lind informed us that Douglas Lind and Klaus Schmidt have independently
obtained results similar to ours in Section~\ref{S-addition}.

{\it Acknowledgements.} I thank George Elliott, Wen Huang and Zhuang Niu for helpful discussion, especially my colleague Jingbo Xia for
discussion about Example~\ref{E-not invertible}.
I also thank Lewis Bowen, Christopher Deninger, Douglas Lind, Andreas Thom, Thomas Ward and the referee for very helpful comments.
Part of this work was carried out while I visited Fields Institute in Summer 2007 and Vanderbilt University in Fall 2009.
I am grateful to Guoliang Yu for his warm hospitality.

\section{Preliminaries} \label{S-preliminary}

\subsection{Background on Entropy Theory} \label{SS-entropy}

In this subsection we recall some background about the entropy theory.
The reader is referred to \cite{Glasner, JMO, Parry, Wal} for detail.
Throughout this paper $\Gamma$ will be a discrete amenable group, unless specified otherwise.
The amenability of $\Gamma$ means that $\Gamma$ has a (left) F{\o}lner net $\{F_n\}_{n\in J}$, i.e.,
each $F_n$ is a nonempty finite subset of $\Gamma$, and $\lim_{n\to \infty}\frac{|KF_n\Delta F_n|}{|F_n|}=0$
for every finite subset $K$ of $\Gamma$ \cite{Pat}.

The following subadditivity
result is known as the Ornstein-Weiss lemma \cite[Theorem 6.1]{LW}.

\begin{proposition}\label{subadditive:prop}
If $\varphi$ is a real-valued function
which is defined on the set of nonempty finite subsets of $\Gamma$ and satisfies
\begin{enumerate}
\item $0\le \varphi(F)<+\infty$,

\item $\varphi(F)\le \varphi(F')$ for all $F\subseteq F'$,

\item $\varphi(F\gamma)=\varphi(F)$ for all nonempty finite $F\subseteq \Gamma$ and $\gamma \in \Gamma$,

\item $\varphi(F\cup F')\le \varphi(F)+\varphi(F')$ if $F\cap F'=\emptyset$,
\end{enumerate}
then $\frac{1}{|F|} \varphi(F)$ converges to some limit $b$
as the set $F$ becomes
more and more (left) invariant in the sense that
for every $\varepsilon>0$ there exist a
nonempty finite set $K\subseteq \Gamma$ and a $\delta>0$ such that
$\big| \frac{1}{|F|} \varphi(F) -b \big| <\varepsilon$
for all nonempty finite
sets $F\subseteq \Gamma$ satisfying $|KF \Delta F|\le \delta |F|$.
\end{proposition}

Let $\alpha$ be an action of $\Gamma$ on a compact Hausdorff space $X$ by homeomorphisms.
For any open cover $\cU$ of $X$, and any nonempty finite subset $F$ of $\Gamma$, set
$\cU^F=\bigvee_{\gamma \in F} \gamma^{-1} \cU$ and denote by $N(\cU)$ the minimal number of elements
in $\cU$ needed to cover $X$. Then the function $F\mapsto \log N(\cU^{F})$
defined on the set of nonempty finite subsets of $\Gamma$ satisfies the conditions in
Proposition~\ref{subadditive:prop}, and hence
$\frac{1}{|F|}\log N(\cU^{F})$ converges
as $F$ becomes more and more (left) invariant. We denote this limit by $\htopol(\alpha, \cU)$.
The {\it topological entropy} of $\alpha$, denoted by $\htopol(\alpha)$,  is defined as the supremum of $\htopol(\alpha, \cU)$ over
all  finite open covers $\cU$ of $X$.

Let $\alpha$ be an action of $\Gamma$ on a probability space
$(X, \cB, \mu)$ by automorphisms. For any finite measurable partition $\cP=\{P_1, \dots, P_k\}$ of $X$,
and any nonempty finite subset $F$ of $\Gamma$, set $\cP^F=\bigvee_{\gamma\in F}\gamma^{-1} \cP$ and $\rH_{\mu}(\cP)=\sum_{j=1}^k-\mu(P_j)\log \mu(P_j)$,
where we take the convention that $0\log 0=0$ so that the function $t\mapsto t\log t$ is continuous for $0\le t\le 1$.
The function $F\mapsto \rH_{\mu}(\cP^{F})$
defined on the set of nonempty finite subsets of $\Gamma$ satisfies the conditions in
Proposition~\ref{subadditive:prop}, and hence
$\frac{1}{|F|}\log \rH_{\mu}(\cP^{F})$ converges
as $F$ becomes more and more (left) invariant. We denote this limit by $\rh_{\mu}(\alpha, \cP)$.
The {\it measure entropy} or {\it Kolmogorov-Sinai entropy} of $\alpha$, denoted by $\rh_{\mu}(\alpha)$,  is defined as the supremum of $\rh_{\mu}(\alpha, \cP)$ over
all finite measurable  partitions $\cP$ of $X$.

A topological space is called a {\it Polish space} if it is separable and admits a compatible complete metric.
A probability space $(X, \cB, \mu)$ is called a {\it standard} if $\cB$ is the Borel $\sigma$-algebra for
some Polish topology on $X$.
Suppose that $(X, \cB, \mu)$ is standard and that $\cB'$ is a
sub-$\sigma$-algebra of $\cB$. Then there is a map $\Eb(\cdot |\cB'): L^1(X, \cB, \mu)\rightarrow L^1(X, \cB', \mu)$, called
the {\it conditional expectation}, determined by
$$\int_A\Eb(f|\cB')(x)\, d\mu(x)=\int_Af(x)\, d\mu(x)$$
 for every $f\in L^1(X, \cB, \mu)$
and $A\in \cB'$. Here one can use either complex or real valued functions for $L^1(X, \cB, \mu)$ and $L^1(X, \cB', \mu)$.
For any $A\in \cB$, one has $0\le \Eb(1_A|\cB')(x)\le 1$ for $\mu$ a.e. $x\in X$, where $1_A$ denotes the characteristic function of $A$. For any finite measurable partition
$\cP$ of $X$, set $\rH_{\mu}(\cP|\cB')=\sum_{P\in \cP}-\int_P\log \Eb(1_P|\cB')(x)\, d\mu(x)$. Now assume further that $\cB'$ is  $\Gamma$-invariant.
Then the function $F\mapsto \rH_{\mu}(\cP^{F}|\cB')$
defined on the set of nonempty finite subsets of $\Gamma$ satisfies the conditions in
Proposition~\ref{subadditive:prop}, and hence
$\frac{1}{|F|}\rH_{\mu}(\cP^{F}|\cB')$ converges
as $F$ becomes more and more (left) invariant. We denote this limit by $\rh_{\mu}(\alpha, \cP|\cB')$.
The {\it conditional entropy of $\alpha$ given $\cB'$}, denoted by $\rh_{\mu}(\alpha|\cB')$,  is defined as the supremum of $\rh_{\mu}(\alpha, \cP|\cB')$ over
all finite measurable  partitions $\cP$ of $X$.

For a compact space $X$, denote by $\cB_X$ the Borel $\sigma$-algebra of $X$.
If $\alpha$ is an action of $\Gamma$ on a compact space $X$ by homeomorphisms, and $\mu$ is a
regular $\Gamma$-invariant Borel probability measure on $X$, then $\alpha$ is also an action of $\Gamma$ on
the probability space $(X, \cB_X, \mu)$ by automorphisms.

Note that every (continuous) automorphism of a compact group preserves the normalized Haar measure.
Thus if $\alpha$ is an action of  $\Gamma$ on a compact group $G$ by automorphisms, it automatically
preserves the normalized Haar measure $\mu$ of $G$. Then we have both the topological entropy $\htopol(\alpha)$
and the measure entropy $\rh_{\mu}(\alpha)$.
It is a result of Deninger that
these two entropies coincide \cite[Theorem 2.2]{Den}. (It was assumed in \cite[Theorem 2.2]{Den} that
$G$ is abelian; but this is not needed.) (The case $\Gamma=\Zb$ was proved by Berg \cite{Berg}; the case $\Gamma=\Zb^d$ was proved
by Lind et al. \cite[page 624]{LSW} \cite[Theorem 13.3]{Sch}.)
Thus we shall denote $\htopol(\alpha)$ and $\rh_{\mu}(\alpha)$  simply
by $\rh(\alpha)$.

\subsection{Background on Group von Neumann Algebras and Fuglede-Kadison Determinants} \label{SS-determinant and algebra}

In this subsection we recall some background about the group von Neumann algebra and the Fuglede-Kadison determinant.

For a Hilbert space $H$, the set $B(H)$ is a $*$-algebra with $T^*$ being the adjoint of $T$, and is equipped with the operator norm
$\|\cdot \|$.
A {\it $C^*$-algebra} is a sub-$*$-algebra of $B(H)$ for some Hilbert space $H$, closed under $\|\cdot \|$.
An element $a$ in $A$ is called {\it positive} and written as $a\ge 0$ if $a=b^*b$ for some $b\in A$.
A tracial state of a unital $C^*$-algebra $A$ is a linear functional $\tr_A: A\rightarrow \Cb$ such that
$\tr_A$ takes value $1$ at the identity of $A$, $|\tr_A(a)|\le \|a\|$ and $\tr_A(ab)=\tr_A(ba)$ for all $a, b\in A$. We refer the reader to \cite{KR, Takesaki} for detail.

In this paper we shall need only three classes of $C^*$-algebras and tracial states. The first class is the $C^*$-algebra
$B(\ell^2_n)$ for each $n\in \Nb$. Each $B(\ell^2_n)$ has a unique tracial state $\tr_{B(\ell^2_n)}$. If we take an orthonormal basis
of $\ell^2_n$ and identify $B(\ell^2_n)$ with $M_n(\Cb)$, then $\tr_{B(\ell^2)}(a)=\frac{1}{n}\sum_{j=1}^na_{j,j}$ for every matrix $a=(a_{i,j})_{1\le i, j\le n}\in M_n(\Cb)$.

Let $\Gamma$ be a discrete amenable group.
The complex group algebra $\Cb\Gamma$ consists of elements in $\Cb^{\Gamma}$ with finite support.
Its multiplication is defined as $(fg)_{\gamma'}=\sum_{\gamma \in \Gamma}f_{\gamma}g_{\gamma^{-1}\gamma'}$ for
all $f, g\in \Cb\Gamma$ and $\gamma \in \Gamma$. We shall also extend this multiplication to the cases like $g\in \Cb[[\Gamma]]$, or
$f\in \Zb\Gamma$ and $g\in (\Rb/\Zb)^{\Gamma}$ whenever it can be defined.
One may identify $\Cb\Gamma$ as a linear subspace of $\ell^2(\Gamma)$ naturally.
For each $f\in \Cb\Gamma$, its left multiplication $g\mapsto fg$ for $g\in \Cb\Gamma$ extends to a bounded linear map of $\ell^2(\Gamma)$.
In this way we shall identify $\Cb\Gamma$ as a subalgebra of $B(\ell^2(\Gamma))$.
It is easily checked that $\Cb\Gamma$ is closed under taking adjoint in $B(\ell^2(\Gamma))$. Explicitly,
$(f^*)_{\gamma}=\overline{f_{\gamma^{-1}}}$ for all $f\in \Cb\Gamma$ and $\gamma\in \Gamma$.
The second class of $C^*$-algebras we need, the left group von Neumann algebra $\cL\Gamma$, is defined as the closure of $\Cb\Gamma$ under the strong operator topology.
Explicitly, $\cL\Gamma$ consists of $T\in B(\ell^2(\Gamma))$ commuting with the right regular representation $\rho$ of $\Gamma$ on $\ell^2(\Gamma)$,
i.e., $(T(h\gamma))_{\gamma'}=(Th)_{\gamma'\gamma}$ for all $h\in \ell^2(\Gamma)$ and $\gamma, \gamma'\in \Gamma$,
where $(h\gamma)_{\gamma''}=h_{\gamma''\gamma}$ for all $\gamma''\in \Gamma$.
The algebra $\cL\Gamma$ has a canonical tracial state $\tr_{\cL\Gamma}$ defined as $\tr_{\cL\Gamma}(a)=\left<ae_{\Gamma}, e_{\Gamma}\right>$.
The trace $\tr_{\cL\Gamma}$ is faithful in the sense that
if $a\in \cL\Gamma$ is positive and $\tr_{\cL\Gamma}(a)=0$, then $a=0$. Throughout this article, we fix this tracial state of $\cL\Gamma$,
and the determinant $\ddet{\cL\Gamma}$ is calculated with respect to it.

Another way to describe the elements of $\cL\Gamma$ is that they are the elements $h$ of $\Cb[[\Gamma]]$ for which the map from
$\Cb\Gamma$ to $\ell^2(\Gamma)$ sending $x$ to $hx$ is well-defined and extends to a bounded linear operator on $\ell^2(\Gamma)$.
It is easy to see that if $h_1$ and $h_2$ are in $\Rb[[\Gamma]]$, then $h_1+ih_2$ is in $\cL\Gamma$ if and only if both $h_1$ and $h_2$ are in
$\cL\Gamma$. It follows that if $h\in \Rb[[\Gamma]]\cap \cL\Gamma$ is invertible in $\cL\Gamma$, then its inverse lies in $\Rb[[\Gamma]]$ and hence preserves $\ell^2_{\Rb}(\Gamma)$.

The third class of $C^*$-algebras we need is the unital commutative $C^*$-algebras. They can be described as unital commutative Banach complex algebras
$A$ with a $*$-operation satisfying $(a^*)^*=a$, $(\lambda a+b)^*=\bar{\lambda}a^*+b^*$, $(ab)^*=b^*a^*$, $\|a^*\|=\|a\|$ and $\|a^*a\|=\|a\|^2$ for all $a, b\in A$
and $\lambda\in \Cb$.

For a tracial state $\tr_A$ of a unital $C^*$-algebra $A$, the {\it Fuglede-Kadison determinant} of an invertible $a\in A$
with respect to $\tr_A$ \cite{FK} is defined as
\begin{equation}
\ddet_{A} a:= \exp(\tr_A \log |a|)=\exp(\frac{1}{2}\tr_A \log (a^*a)),
\end{equation}
where $|a|=(a^*a)^{1/2}$ is the absolute part of $a$. (The Fuglede-Kadison determinant is also defined for noninvertible elements of $A$, but the definition is more involved.) For detail and application of the Fuglede-Kadison determinant to $L^2$-invariants, see
\cite{Luck}.

For any $n\in \Nb$ and any invertible $a\in B(\ell^2_n)$, one
has $\ddet_{B(\ell^2_n)}(a)=|\det a|^{1/n}$.

Among many nice properties of the Fuglede-Kadison determinant, we shall need the following ones:

\begin{theorem}\cite[Lemma 1, Theorem 1]{FK} \label{T-FK}
Let $\tr$ be a tracial state of a unital $C^*$-algebra $A$. Then
\begin{enumerate}
\item for any invertible  $a\in A$, one has $\ddet_{A}(a)=\ddet_{A}(a^*)$;

\item for any $0\le a\le b$ in $A$ with $a$ being invertible in $A$, one has $\ddet_{A}a\le \ddet_{A}b$.
\end{enumerate}
\end{theorem}

\section{Finite Group Case} \label{S-finite group}

In this section we prove Theorem~\ref{T-main} for the case $\Gamma$ is finite. This case is easily proved and appeared
in \cite[Section 7]{Den}.
However, we choose to give a proof of this case here, since it reveals the essence of the equality in Theorem~\ref{T-main}.

The following lemma is well known \cite[Lemma 4]{Solomyak}. For the convenience of the reader, we give  a proof.

\begin{lemma} \label{L-determinant quotient group size}
Let $n\in \Nb$ and let $T:\Cb^n\rightarrow \Cb^n$ be an invertible linear map, preserving $\Zb^n$.
Then $|\det T|=|\Zb^n/T\Zb^n|$.
\end{lemma}
\begin{proof}
Note that  $T\Zb^n$ has rank $n$.
By the elementary divisor theorem \cite[Theorem III.7.8]{Lang}
there are a basis $e_1, \dots, e_n$ of $\Zb^n$ and nonzero integers $k_1, \dots, k_n$ such that
$k_1e_1, \dots, k_ne_n$ is a basis of $T\Zb^n$. Since $Te_1, \dots, Te_n$ is also a basis of
$T\Zb^n$, there exists $Q\in \GL_n(\Zb)$ with $(Te_1, \dots, Te_n)=(k_1e_1, \dots, k_ne_n)Q$.
Then the matrix of $T$ under the basis
$e_1, \dots, e_n$ is $\diag(k_1, \dots, k_n)\cdot Q$. Thus
\begin{eqnarray}
|\det T|&=&|\det (\diag(k_1, \dots, k_n)\cdot Q)|=|\prod_{1\le j\le n}k_j|=|\Zb^n/T\Zb^n|.
\end{eqnarray}
\end{proof}


Let $\Gamma$ be a discrete amenable group and let $f\in \Zb\Gamma$.
The canonical pairing between $\Zb\Gamma$ and its Pontryagin dual $\widehat{\Zb\Gamma}=(\Rb/\Zb)^{\Gamma}$ is given by
$$ \left< g, h\right>=\sum_{\gamma \in \Gamma}g_{\gamma}h_{\gamma}$$
for all $g\in \Zb\Gamma$ and $h\in (\Rb/\Zb)^{\Gamma}$.
It is easy to check that
$$ \left< gf, h\right>=\left<g, hf^*\right>$$
for all $g\in \Zb\Gamma$ and $h\in (\Rb/\Zb)^{\Gamma}$.
It follows that
$X_f=\{h\in (\Rb/\Zb)^{\Gamma}: hf^*=0\}$ and $\alpha_f$ is the restriction of the left shift action of $\Gamma$
on $(\Rb/\Zb)^{\Gamma}$ to $X_f$. For $h\in (\Rb/\Zb)^{\Gamma}$, denote by $\tilde{h}$ the ``adjoint" element in $(\Rb/\Zb)^{\Gamma}$ defined as
$\tilde{h}_{\gamma}=h_{\gamma^{-1}}$ for all $\gamma\in \Gamma$. Note that the map $h\mapsto \tilde{h}$ is an automorphism of the compact group
$(\Rb/\Zb)^{\Gamma}$, and intertwines the left and right shift actions of $\Gamma$. The image of $X_f$ under this map is $\{h\in (\Rb/\Zb)^{\Gamma}:fh=0\}$. In the rest of this paper, we shall write
\begin{eqnarray} \label{E-X}
 X_f= \{h\in (\Rb/\Zb)^{\Gamma}:fh=0\},
\end{eqnarray}
and under this identification $\alpha_f$ is the restriction of the right shift action of $\Gamma$ on $(\Rb/\Zb)^{\Gamma}$ to $X_f$.

\begin{theorem} \label{T-finite}
Let $\Gamma$ be a finite group and let $f\in \Zb\Gamma$ be invertible in $\cL\Gamma$. Then
$$ \rh(\alpha_f)=\frac{1}{|\Gamma|}\log |X_f|=\frac{1}{|\Gamma|}\log |\Zb\Gamma/f\Zb\Gamma|=\frac{1}{|\Gamma|}\log |\det f|=\log \ddet_{\cL\Gamma} f.$$
\end{theorem}
\begin{proof} From the definition of topological entropy, we have $\rh(\alpha_f)=\frac{1}{|\Gamma|}\log |X_f|$.

Note that both $X_f$ and $\Zb\Gamma/ f\Zb\Gamma$ are abelian groups. We claim that they are isomorphic.
Writing $(\Rb/\Zb)^{\Gamma}$ as $\Rb \Gamma/\Zb \Gamma$, we may identify $X_f$ with $\{g\in \Rb\Gamma: fg\in \Zb\Gamma\}/\Zb\Gamma$.
Since the left multiplication by $f$ restricts to a group automorphism of  $\Rb\Gamma$ and sends $\Zb\Gamma$ onto $f\Zb\Gamma$, the claim is proved.
It follows that $|X_f|= |\Zb\Gamma/f\Zb\Gamma|$.

By Lemma~\ref{L-determinant quotient group size} one has $ |\Zb\Gamma/f\Zb\Gamma|=|\det f|$.

Note that the unique tracial state of $B(\ell^2(\Gamma))$ restricts to the canonical trace of $\cL\Gamma$.
Thus $\ddet_{\cL\Gamma}f=\ddet_{B(\ell^2(\Gamma))}f=|\det f|^{\frac{1}{|\Gamma|}}$.
\end{proof}

\begin{notation} \label{N-compression}
For any
nonempty finite subset $F$ of $\Gamma$, denote by $p_F$
the restriction map $\Cb[[\Gamma]]\rightarrow \Cb[F]$, and by $\iota_F$ the embedding $\Cb[F]\rightarrow \ell^2(\Gamma)$. For $f\in \cL\Gamma$, set
$f_F:=p_F\circ f\circ \iota_F\in B(\Cb[F])$.
\end{notation}

Now consider the case $\Gamma$ is infinite countable.
Let $\{F_n\}_{n\in \Nb}$ be a (left) F{\o}lner sequence of $\Gamma$, and let
$f\in \Zb\Gamma$ be invertible in $\cL\Gamma$.
Since $F_n$ is the analogue of a finite group,
 the analogue of Theorem~\ref{T-finite} is
\begin{eqnarray*}
\rh(\alpha_f)&=&\lim_{\varepsilon\to 0}\frac{1}{|F_n|}\log s_{F_n, \infty}(\varepsilon)=\frac{1}{|F_n|}\log |\Zb[F_n]/f_{F_n}\Zb[F_n]|\\
\nonumber &=&\frac{1}{|F_n|}\log |\det f_{F_n}|=\log \ddet_{\cL\Gamma} f
\end{eqnarray*}
for each $n\in \Nb$,  where $s_{F_n, \infty}(\varepsilon)$ is the cardinality of certain set resembling $X_f$ restricted to $F_n$ and will
be defined at the beginning of Section~\ref{S-positive}. On the other hand, $F_n$ approximates $\Gamma$ as $n\to \infty$. Thus a more precise
and reasonable analogue of Theorem~\ref{T-finite} is
\begin{eqnarray} \label{E-intuition}
\rh(\alpha_f)&=&\lim_{\varepsilon\to 0}\lim_{n\to \infty}\frac{1}{|F_n|}\log s_{F_n, \infty}(\varepsilon)=\lim_{n\to \infty}\frac{1}{|F_n|}\log |\Zb[F_n]/f_{F_n}\Zb[F_n]|\\
\nonumber &=&\lim_{n\to \infty}\frac{1}{|F_n|}\log |\det f_{F_n}|=\log \ddet_{\cL\Gamma} f.
\end{eqnarray}
Indeed, this is the intuition behind Theorem~\ref{T-main}. But there is some immediate difficulty even for making sense of (\ref{E-intuition}). For instance, $f_{F_n}$ may fail to be invertible. In such case, $|\Zb[F_n]/f_{F_n}\Zb[F_n]|=\infty$ and $\det f_{F_n}=0$.

\section{$\ell^p$-version of R. Bowen's Definition of Topological Entropy} \label{S-Bowen}

In this section we prove Theorem~\ref{T-sep}, providing an $\ell^p$-version
of R. Bowen's definition of topological entropy. Throughout this section $\Gamma$ is a discrete amenable group.

Let $\alpha$ be an action of $\Gamma$ on a compact Hausdorff space $X$ by homeomorphisms.
Recall that a {\it continuous pseudometric} on $X$ is a symmetric continuous map $X\times X\rightarrow \Rb_+$,
vanishing on the diagonal of $X\times X$ and satisfying the triangle inequality.
Denote by $\cM$ the set of all continuous pseudometrics on $X$.
Let $\vartheta \in \cM$.
For a nonempty
finite subset $F\subseteq \Gamma$, $1\le p\le \infty$ and $x, y\in
X$, denote by $d_{\vartheta, F, p}(x, y)$ the quotient of the $\ell^p$-norm of the function
$\gamma\mapsto \vartheta(\gamma x, \gamma y)$ on $F$ divided by
$|F|^{1/p}$. We say
that $E\subseteq X$ is {\it $[\vartheta, F, p, \varepsilon]$-separated} if
for any $x\neq y$ in $E$, $d_{\vartheta, F, p}(x, y)> \varepsilon$.
We say that $E\subseteq X$ is {\it $[\vartheta, F, p, \varepsilon]$-spanning}
if for any $x\in X$, there is some $y\in E$ with $d_{\vartheta, F, p}(x,
y)\le \varepsilon$. Denote by $s_{\vartheta, F, p}(\varepsilon)$
($r_{\vartheta, F, p}(\varepsilon)$ resp.) the maximal (minimal resp.)
cardinality of $[\vartheta, F, p, \varepsilon]$-separated ($[\vartheta, F, p,
\varepsilon]$-spanning resp.) subsets of $X$.

\begin{lemma} \label{compare:lemma}
Let $\alpha$ be an action of $\Gamma$ on a compact Hausdorff space $X$ by homeomorphisms.
Let $\vartheta$ be a continuous pseudometric of $X$.
For any $\varepsilon>0$, $\lambda>1$, and $1\le p<\infty$,
there exists some $\varepsilon'>0$ such that
$\lambda^{|F|}s_{\vartheta, F, p}(\varepsilon')\ge s_{\vartheta, F, \infty}(\varepsilon)$
for all nonempty finite subsets $F$ of $\Gamma$.
\end{lemma}
\begin{proof}
Cover $X$ by finitely many, say $M$, closed $\vartheta$-balls of radius $\varepsilon/2$.
By Stirling's formula there is some $c\in (0, 1/2)$ such that
$\binom{n}{cn}\le \lambda^{n/2}$ for all $n\in \Nb$.
We may assume that $M^c\le \lambda^{1/2}$.
Set
$\varepsilon'=c^{\frac{1}{p}}\varepsilon/2$.

Let $F$ be a nonempty finite subset of $\Gamma$ and let $E$ be a
$[\vartheta, F, \infty, \varepsilon]$-separated subset of $X$ with $|E|=s_{\vartheta, F,
\infty}(\varepsilon)$. For each $x\in E$ denote by $B(x,
\varepsilon/2)$ the set of elements $y$ in $E$ such that
$|\{\gamma \in F: \vartheta(\gamma x, \gamma y)>
\varepsilon/2\}|<c|F|$.
If $x$ and  $y$ are in $E$ and $y\not\in B(x, \varepsilon/2)$, then
$$ d_{\vartheta, F, p}(x, y)>\frac{((\varepsilon/2)^pc|F|)^{1/p}}{|F|^{1/p}}=(\varepsilon/2)c^{1/p}=\varepsilon'.$$
Take a subset $E'$ of $E$ maximal with
respect to the property that for any $x\neq y$ in $E'$, $y\notin
B(x, \varepsilon/2)$. Then $\bigcup_{x\in E'}B(x, \varepsilon/2)=E$
and $E'$ is $[\vartheta, F, p, \varepsilon']$-separated. Denote by $D$ the
maximum of $|B(x, \varepsilon/2)|$ over all $x\in E$. Then
$D|E'|\ge |E|$. Thus it suffices to show that $\lambda^{|F|}\ge
D$.

Fix $x\in E$. For any $y\in B(x, \varepsilon/2)$ there is some
$K_y\subseteq F$ with $|K_y|=\lfloor c|F|\rfloor$ and
$\vartheta(\gamma x, \gamma y)\le \varepsilon/2$ for all
$\gamma \in F\setminus K_y$, where $\lfloor t\rfloor$ denotes the
largest integer no bigger than $t$. Then there are a subset $B'$
of $B(x, \varepsilon/2)$ with $|B'|\ge |B(x,
\varepsilon/2)|/\binom{|F|}{c|F|}$ and a subset $K$ of $F$ with
$|K|=\lfloor c|F|\rfloor$ such that $K_y=K$ for all $y\in B'$.
Then $\vartheta(\gamma y, \gamma z)\le \vartheta(\gamma y, \gamma x)+\vartheta(\gamma x, \gamma z)\le \varepsilon$ for all $y, z\in B'$ and $\gamma \in F\setminus K$.
Note that, as a subset of $E$, $B'$ is $[\vartheta, F, \infty, \varepsilon]$-separated.
It
follows that for any $y\neq z$ in $B'$ there is some $\gamma$ in
$K$ with $\vartheta(\gamma y, \gamma z)> \varepsilon$. Then
$\gamma y$ and $\gamma z$ must lie in different closed balls
which we take at the beginning of the proof. Consequently,
$|B'|\le M^{|K|}$. Therefore
$$ |B(x, \varepsilon/2)|\le |B'|\binom{|F|}{c|F|}\le
M^{c|F|}\lambda^{|F|/2}\le \lambda^{|F|}.$$
This finishes the proof of the lemma.
\end{proof}

We say that an open subset $U$ of $X$ is {\it generated by $\vartheta$} if $U$ is in the weakest
topology of $X$ making $\vartheta$ continuous, i.e., $U$ is a union of open $\vartheta$-balls with
positive radii.
We say that a finite open cover $\cU=\{U_1, \dots, U_n\}$ of $X$ is {\it generated by
$\vartheta$} if each $U_j$ is generated by $\vartheta$.
For any nonempty finite subset $F$ of $\Gamma$, we define $\vartheta^F\in \cM$ by setting
$\vartheta^F(x, y)=\max_{\gamma \in F}\vartheta(\gamma x, \gamma y)$ for all $x, y\in X$.
We say that an open subset $U$ of $X$ is {\it generated by $\vartheta$ under $\alpha$}
if
$U$ is contained in the weakest topology on $X$ making all the pseudometrics $(x, y)\mapsto \vartheta(\gamma x, \gamma y)$ continuous,
equivalently, $U$ is a union of open sets $U_F$ generated by $\vartheta^F$ for $F$ running over nonempty
finite subsets of $\Gamma$. We say that the topology of $X$ is {\it generated by $\vartheta$ under $\alpha$}
if the topology on $X$ is exactly the weakest topology making all the pseudometrics $(x, y)\mapsto \vartheta(\gamma x, \gamma y)$ continuous.
Having zero $\vartheta$-distance is an equivalence relation on $X$.
For $x\in X$ denote by $[x]$ its equivalence class.
Denote by $X_{\vartheta}$ the quotient space of $X$ consisting of all such equivalence classes, equipped with the quotient topology.
Then $\vartheta$ induces a metric on $X_{\vartheta}$. Equip $(X_{\vartheta})^{\Gamma}$ with the right shift action of $\Gamma$.
It is easy to see that the topology of $X$ is generated by $\vartheta$ under $\alpha$
if and only
if the natural $\Gamma$-equivariant continuous map $X\rightarrow (X_{\vartheta})^{\Gamma}$ sending
 $x$ to $\gamma\mapsto [\gamma x]$ is an embedding, if and only if any two points $x$ and $y$ of $X$ are equal exactly when $\vartheta(\gamma x, \gamma y)=0$ for all $\gamma \in \Gamma$.
We say that a finite open cover $\cU=\{U_1, \dots, U_n\}$ of $X$ is {\it generated by
$\vartheta$ under $\alpha$}
if each $U_j$ is so.

The case $p=\infty$ and $\Gamma=\Zb^d$ of the following theorem was proved by Schmidt \cite[Proposition 13.7]{Sch},
and the case $p=\infty$ for general $\Gamma$ was proved by
Deninger \cite[Proposition 2.3]{Den}.
For completeness we include also a proof for the case $p=\infty$ here.

\begin{theorem} \label{T-sep}
Let $\alpha$ be an action of $\Gamma$ on a compact Hausdorff space $X$ by homeomorphisms.
Let $\vartheta$ be a continuous pseudometric of $X$.
Let $\{F_n\}_{n\in J}$ be a (left) F{\o}lner net of
$\Gamma$.
For any $1\le p \le \infty$, we have
\begin{eqnarray*}
 \sup_{\cU}\htopol(\alpha, \cU)&=&\lim_{\varepsilon \to 0}\limsup_{n\to \infty}\frac{1}{|F_n|}
\log s_{\vartheta, F_n, p}(\varepsilon)
=\lim_{\varepsilon \to 0}\liminf_{n\to \infty}\frac{1}{|F_n|}
\log s_{\vartheta, F_n, p}(\varepsilon)\\
&=&  \lim_{\varepsilon \to 0}\limsup_{n\to \infty}\frac{1}{|F_n|}
\log r_{\vartheta, F_n, p}(\varepsilon)
=\lim_{\varepsilon \to 0}\liminf_{n\to \infty}\frac{1}{|F_n|}
\log r_{\vartheta, F_n, p}(\varepsilon),
\end{eqnarray*}
where $\cU$ runs through all finite open covers of $X$ generated by $\vartheta$ under $\alpha$.
In particular, if the topology of $X$ is generated by $\vartheta$ under $\alpha$,
then we have
\begin{eqnarray*}
 \htopol(\alpha)&=&\lim_{\varepsilon \to 0}\limsup_{n\to \infty}\frac{1}{|F_n|}
\log s_{\vartheta, F_n, p}(\varepsilon)
=\lim_{\varepsilon \to 0}\liminf_{n\to \infty}\frac{1}{|F_n|}
\log s_{\vartheta, F_n, p}(\varepsilon)\\
&=&  \lim_{\varepsilon \to 0}\limsup_{n\to \infty}\frac{1}{|F_n|}
\log r_{\vartheta, F_n, p}(\varepsilon)
=\lim_{\varepsilon \to 0}\liminf_{n\to \infty}\frac{1}{|F_n|}
\log r_{\vartheta, F_n, p}(\varepsilon).
\end{eqnarray*}
\end{theorem}
\begin{proof}
We prove first the theorem for $p=\infty$. Note that
$$r_{\vartheta, F, \infty}(\varepsilon)\le s_{\vartheta, F, \infty}(\varepsilon)\le r_{\vartheta, F,
\infty}(\varepsilon/2).$$
Thus it suffices to
show that $\sup_{\cU}\htopol(\alpha, \cU)\ge \lim_{\varepsilon \to 0}\limsup_{n\to \infty}\frac{1}{|F_n|}
\log s_{\vartheta, F_n, \infty}(\varepsilon)$
and $\sup_{\cU}\htopol(\alpha, \cU)\le \lim_{\varepsilon \to 0}\liminf_{n\to \infty}\frac{1}{|F_n|}
\log s_{\vartheta, F_n, \infty}(\varepsilon)$.

Let $\varepsilon>0$. Take a finite open cover $\cU$ of $X$ consisting of
open $\vartheta$-balls with radius $\varepsilon/2$.
Then $\cU$ is generated by $\vartheta$.
We have $s_{\vartheta, F, \infty}(\varepsilon)\le N(\cU^F)$ for every nonempty finite subset $F$ of $\Gamma$,
and hence  $\limsup_{n\to \infty}\frac{1}{|F_n|}
\log s_{\vartheta, F_n, \infty}(\varepsilon)\le \htopol(\alpha, \cU)$.
Therefore $\lim_{\varepsilon \to 0}\limsup_{n\to \infty}\frac{1}{|F_n|}
\log s_{\vartheta, F_n, \infty}(\varepsilon)\le \sup_{\cU}\htopol(\alpha, \cU)$.

Let $\cU$ be a finite open cover of $X$ generated by $\vartheta$ under $\alpha$.
Then we can find a finite open cover $\cV$ of $X$ finer than $\cU$ such that
$\cV$ is generated by $\vartheta^K$ for some nonempty finite subset $K$ of $\Gamma$.
It follows that there exists some $\varepsilon>0$
such that every open $\vartheta^K$-ball with radius $3\varepsilon$ is contained in some element of $\cV$.
Cover $X$ by finitely many, say $M$, open $\vartheta$-balls with radius $\varepsilon$.
We have
$$M^{|KF\setminus F|}r_{\vartheta, F, \infty}(\varepsilon)\ge r_{\vartheta, KF, \infty}(2\varepsilon)\ge N(\cV^{F})\ge N(\cU^F)$$
for every nonempty finite subset $F$ of $\Gamma$,
and hence  $\liminf_{n\to \infty}\frac{1}{|F_n|}
\log r_{\vartheta, F_n, \infty}(\varepsilon)\ge \htopol(\alpha, \cU)$.
Therefore $\lim_{\varepsilon \to 0}\liminf_{n\to \infty}\frac{1}{|F_n|}
\log r_{\vartheta, F_n, \infty}(\varepsilon)\ge \sup_{\cU}\htopol(\alpha, \cU)$.
This proves the case $p=\infty$.

Now the case $1\le p<\infty$ follows
from the case $p=\infty$, the facts $s_{\vartheta, F,
p}(\varepsilon)\le s_{\vartheta, F, \infty}(\varepsilon)$ and $r_{\vartheta, F,
p}(\varepsilon)\le s_{\vartheta, F, p}(\varepsilon)\le r_{\vartheta, F,
p}(\varepsilon/2)$, and Lemma~\ref{compare:lemma}.
\end{proof}

\section{Positive Case} \label{S-positive}

In this section we show that the intuitive equalities (\ref{E-intuition}) do hold when $f$ is positive (Theorem~\ref{T-positive}).
This proves Theorem~\ref{T-main} in such case.
Throughout this section $\Gamma$ is a discrete amenable group.

Denote by $\vartheta$ the metric
on $\Rb/\Zb$ induced from the standard metric on $\Rb$, i.e. $\vartheta(t\mod \Zb,\, \, t'\mod \Zb)=\min_{m\in \Zb}|t-t'-m|$.
Recall the identification (\ref{E-X}). Via the projection $X_f\rightarrow \Rb/\Zb$ sending $x$ to $x_{e_{\Gamma}}$, we shall think of
$\vartheta$ as a continuous pseudometric on $X_f$. Clearly the topology of $X_f$ is
generated by $\vartheta$ under $\alpha_f$. Thus we can apply Theorem~\ref{T-sep}.
We shall make use of the cases $p=2$ and $p=\infty$. We shall abbreviate $s_{\vartheta, F, p}(\varepsilon)$
as $s_{F, p}(\varepsilon)$ etc.

The following result is crucial for the comparison of $s_{F, p}(\varepsilon)$, $r_{F, p}(\varepsilon)$ and $|\Zb[F_n]/f_{F_n}\Zb[F_n]|$.

\begin{lemma} \label{ball:lemma}
There exists some universal constant $C>0$ such that for any
$\lambda>1$,
there is some $\delta>0$ so that for any nonempty finite
set $Y$, any positive integer $n$ with $|Y|\le \delta n$, and any $M\ge 1$ one has
$$ |\{x\in \Zb[Y]: \|x\|_2\le M\cdot n^{1/2}\}|\le
C\lambda^{n}M^{|Y|}.$$
\end{lemma}
\begin{proof}
Let $\delta>0$ be a small number less than $e^{-1}$, which we
shall determine later.
Let $Y$ be a nonempty finite set and $n$ be a positive integer
with $|Y|\le \delta n$.  For each $x\in \Zb[Y]$, denote $\{z\in
\Rb[Y]: 0\le z_{y}-x_{y}\le 1 \mbox{ for all } y \in Y\}$ by
$D_x$. Denote $\{x\in \Zb[Y]: \|x\|_2\le M\cdot n^{1/2}\}$ by $S$
and denote the union of $D_x$ for all $x\in S$ by $D_S$. Then the
(Euclidean) volume of $D_S$ is equal to $|S|$. Note that
$\|z\|_2\le M\cdot n^{1/2}+n^{1/2}\le 2Mn^{1/2}$ for every $z\in D_S$.

A simple calculation shows that the function
$\varsigma(t):=(n/t)^{t/2}$ is increasing for $0<t\le ne^{-1}$.
The volume of the unit ball of $\Rb[Y]$ under $\| \cdot \|_2$ is $\pi^{|Y|/2}/(|Y|/2)!$
\cite[page 9]{CS}. By Stirling's formula there exists some
constant $C'>0$ such that $m!\ge C' \sqrt{m}(\frac{m}{e})^m$ for
all $m\ge 1$. Thus the volume of $D_S$ is no bigger than
\begin{eqnarray*}
(\pi^{|Y|/2}(2Mn^{1/2})^{|Y|})/(|Y|/2)! &\le&
(\pi^{|Y|/2}(2Mn^{1/2})^{|Y|})/(C'\sqrt{|Y|/2}(|Y|/(2e))^{|Y|/2}) \\
&\le& CC_1^{|Y|}(n/|Y|)^{|Y|/2}M^{|Y|}= CC_1^{|Y|}\varsigma(|Y|)M^{|Y|}\\
&\le& CC_1^{\delta n}\varsigma(\delta n)M^{|Y|}=CC_1^{\delta
n}\delta^{-\delta n/2}M^{|Y|},
\end{eqnarray*}
where $C=\sqrt{2}/C'$ and $C_1=2\sqrt{2e\pi}$. Take $\delta>0$ so
small that $C_1^{\delta}\delta^{-\delta/2}\le \lambda$. Then the
volume of $D_S$ is no bigger than $C\lambda^{n}M^{|Y|}$. Consequently,
$|S|\le C\lambda^{n}M^{|Y|}$.
\end{proof}

We need the following result of Deninger (note that the assumption
in \cite[Corollary 3.4]{Den} that $\Gamma$ is finitely generated
is not needed). In Corollary~\ref{C-determinant formula} we shall generalize the equality part to
non-positive elements in the presence of perturbations. Recall the notations $p_F$ and $f_F$ in Notation~\ref{N-compression}.

\begin{lemma}\cite[Theorem 3.2, Proposition 3.3, Corollary 3.4]{Den} \label{L-Deninger}
Let $f\in \Zb\Gamma$ be invertible  and positive in $\cL\Gamma$. Then
$f_F$ is invertible and $\|(f_F)^{-1}\|\le \|f^{-1}\|$ for every nonempty finite subset $F$ of
$\Gamma$, and
$$ \ddet_{\cL\Gamma} f=\lim_{n\to \infty}|\det f_{F_n}|^{1/|F_n|}=\lim_{n\to \infty}|\Zb[F_n]/f_{F_n}\Zb[F_n]|^{1/|F_n|}$$
for any (left) F{\o}lner net $\{F_n\}_{n\in J}$ of $\Gamma$.
\end{lemma}

\begin{notation} \label{N-support}
For $f\in \Cb\Gamma$, denote by $K_f$ the union of the supports of $f$ and $f^*$,  and
the identity of $\Gamma$.
\end{notation}

\begin{lemma} \label{L-positive small}
Let $f\in \Zb\Gamma$ be invertible and positive in $\cL\Gamma$. Then
for any $\lambda>1$ and $\varepsilon>0$, there is some $\delta>0$
such that when a nonempty finite subset $F\subseteq \Gamma$
satisfies $|K^2_fF\setminus F|\le \delta|F|$ we have
$$ s_{F, 2}(\varepsilon)\le C\lambda^{|F|}|\Zb[F]/f_{F}\Zb[F]|,$$
where $C$ is the universal constant in Lemma~\ref{ball:lemma}.
\end{lemma}
\begin{proof} Write $K$ for $K_f$.
Take $1>\delta>0$ such that $(\|f^{-1}\|\cdot \|f\|\cdot 2^{1/2})^{\delta}\le \lambda^{1/2}$
and $\delta^{1/2}\|f^{-1}\|\cdot \|f\|_1\le \varepsilon$,
and that $\delta$
satisfies the conclusion of
Lemma~\ref{ball:lemma} for $\lambda'=\lambda^{1/2}$.
Let $F$ satisfy the hypothesis.

Take an $[F, 2, \varepsilon]$-separated subset $E\subseteq X_f$
with $|E|=s_{F, 2}(\varepsilon)$. For each $x\in E$ denote by
$\tilde{x}$ the element in $[0, 1)^{\Gamma}$ such that $x$ is the
image of $\tilde{x}$ under the natural map $[0,
1)^{\Gamma}\rightarrow (\Rb/\Zb)^{\Gamma}$. Then
$f\tilde{x}\in \Zb[[\Gamma]]$ and hence
$p_{F}(f\tilde{x})\in \Zb[F]$. Denote by $\varphi_{F}$ the
quotient map $\Zb[F]\rightarrow \Zb[F]/f_{F}\Zb[F]$. We get a map
$\psi: E\rightarrow \Zb[F]/f_{F}\Zb[F]$ sending $x$ to
$\varphi_{F}(p_{F}(f\tilde{x}))$. It suffices to
show that for any $a\in \Zb[F]/f_F\Zb[F]$, the preimage of $a$
under $\psi$ has at most $C\lambda^{|F|}$ elements. Fix $a\in \Zb[F]/f_F\Zb[F]$
and $y\in \psi^{-1}(a)$.

For each $x\in E$, set $x'=p_{KF}(\tilde{x})$. We shall identify $\Cb[KF]$ naturally as a subspace
of $\ell^2(\Gamma)$ via the embedding $\iota_{KF}$ in Notation~\ref{N-compression}. Note that
$\psi(x)=\varphi_F(p_F(fx'))$.  Suppose that $x\in \psi^{-1}(a)$. Then $p_F(f(x'-y'))$ lies in $f_F\Zb[F]$, and hence
\begin{align} \label{E-positive}
p_F(f(x'-y'))=f_F(h_x)
\end{align}
for some $h_x\in \Zb[F]$. Set $z_x=f(x'-y')-fh_x$. Then
$$ p_F(z_x)=p_F(f(x'-y')-fh_x)=p_F(f(x'-y'))-f_F(fh_x)=0.$$
Thus $z_x$ is in $\Rb\Gamma$ and vanishes on $F$, and
\begin{eqnarray} \label{positive:eq}
f(x'-y')=fh_x+z_x.
\end{eqnarray}
By Lemma~\ref{L-Deninger} the linear operator $f_{F}$ is invertible
and $\| (f_{F})^{-1}\|\le \|f^{-1}\|$. From \eqref{E-positive} we get
$$ h_x=(f_{F})^{-1}(p_{F}(f(x'-y'))).$$
Thus
\begin{eqnarray*}
\|h_x\|_2&\le &\|(f_F)^{-1}\|\cdot\|f\| \cdot \|x'-y'\|_2\le \|f^{-1}\|\cdot \|f\|\cdot \|x'-y'\|_{\infty}\cdot |KF|^{1/2}\\
&\le &\|f^{-1}\|\cdot \|f\|\cdot |KF|^{1/2}\le \|f^{-1}\|\cdot \|f\|\cdot 2^{1/2}\cdot |F|^{1/2},
\end{eqnarray*}
and hence
$$\|p_{FK\setminus F}(f h_x)\|_2\le \|f\|\cdot \|h_x\|_2\le \|f^{-1}\|\cdot \|f\|^2\cdot 2^{1/2}\cdot |F|^{1/2}.$$
By Lemma~\ref{ball:lemma} one has
\begin{eqnarray*}
|\{p_{KF\setminus F}(f h_x): x\in \psi^{-1}(a)\}|&\le &C\lambda^{|F|/2}(\|f^{-1}\|\cdot \|f\|^2\cdot 2^{1/2})^{|KF\setminus F|}\\
&\le & C\lambda^{|F|/2}(\|f^{-1}\|\cdot \|f\|^2\cdot 2^{1/2})^{\delta|F|}\\
&\le & C\lambda^{|F|}.
\end{eqnarray*}
Thus we can find a subset $W\subseteq
\psi^{-1}(a)$ with $C\lambda^{|F|}|W|\ge |\psi^{-1}(a)|$ such that
$ p_{KF\setminus F}(f h_{x_1})=p_{KF\setminus
F}(f h_{x_2})$ for all $x_1, x_2\in W$. Let $x_1, x_2\in W$. Applying
(\ref{positive:eq}) to $x=x_1$ and $x=x_2$ respectively, we get
$$ f(x'_1-x'_2)=f(x'_1-y')-f(x'_2-y')=f(h_{x_1}-h_{x_2})+(z_{x_1}-z_{x_2}).$$
Since $f(h_{x_1}-h_{x_2})$ has support contained in $F$, while $z_{x_1}-z_{x_2}$ has support contained in $K^2F\setminus F$, one has
\begin{eqnarray*}
 \|z_{x_1}-z_{x_2}\|_2&=& \|p_{K^2F\setminus
F}(f(x'_1-x'_2))\|_2
\le \|f(x'_1-x'_2)\|_{\infty}\cdot |K^2F\setminus
F|^{1/2}\\
&\le & \|f\|_1\cdot \|x'_1-x'_2\|_{\infty}\cdot |K^2F\setminus
F|^{1/2}
\le \delta^{1/2} \|f\|_1\cdot |F|^{1/2},
\end{eqnarray*}
and hence
\begin{eqnarray*}
\|p_F(f^{-1}(z_{x_1}-z_{x_2}))\|_2
&\le & \|f^{-1}(z_{x_1}-z_{x_2})\|_2 \le  \|f^{-1}\|\cdot \|z_{x_1}-z_{x_2}\|_2 \\
&\le & \delta^{1/2}\|f^{-1}\|\cdot \|f\|_1\cdot |F|^{1/2}\le
\varepsilon |F|^{1/2}.
\end{eqnarray*}
If $x_1\neq x_2$,
then
$$\|p_F(f^{-1}(z_{x_1}-z_{x_2}))\|_2=\|p_F((x'_1-x'_2)-(h_{x_1}-h_{x_2}))\|_2\ge
d_{F, 2}(x_1, x_2)|F|^{1/2}> \varepsilon |F|^{1/2},$$ which is a
contradiction. Therefore $W$ contains at most one point. Thus
$$ |\psi^{-1}(a)|\le C\lambda^{|F|}|W|\le C\lambda^{|F|},$$
as desired.
\end{proof}

For an abelian group $G$, denote by $G_{\rm tor}$ the subgroup of
torsion elements. If $f\in \Zb\Gamma$ and $f_F$ is invertible for some nonempty finite subset $F$ of $\Gamma$,
then $f_F\Zb[F]$ has rank $|F|$, and hence $\Zb[F]/f_F\Zb[F]$ is a finite group. In the case,
we shall apply the following result.

\begin{lemma} \label{L-positive large}
Let $f\in \Zb\Gamma$ be invertible in $\cL\Gamma$.
Then for any $\lambda>1$, there is some $\delta>0$ such that for
any nonempty finite subset
$F\subseteq \Gamma$ satisfying $|K_fF\setminus F|\le \delta|F|$ we
have
$$ C\lambda^{|F|}s_{F, \infty}(\frac{1}{2\|f\|_1})\ge |(\Zb[F]/f_{F}\Zb[F])_{\rm tor}|,$$
where $C$ is the universal constant in Lemma~\ref{ball:lemma}.
\end{lemma}
\begin{proof} Write $K$ for $K_f$. Set $D=4\|f\|_1$ and $\varepsilon=2D^{-1}$.
Take $\delta>0$ such that $(D\cdot \|f\|\cdot \|f^{-1}\|)^{\delta}\le \lambda^{1/2}$,
and that $\delta$
satisfies the conclusion of
Lemma~\ref{ball:lemma} for $\lambda'=\lambda^{1/2}$. Let $F$ satisfy the hypothesis.

Denote $\Zb[F]/f_F\Zb[F]$ by $G$. Let $x\in G_{\rm tor}$. Take
$\tilde{x}\in \Zb[F]$ such that the image of $\tilde{x}$ in $G$ under the quotient
map $\Zb[F]\rightarrow G$ is
equal to $x$. Then
$$ k\tilde{x}=f_{F}w$$
for some positive integer $k$ and some $w\in \Zb[F]$. Write
$\frac{1}{k}w$ as $w_1+w_2$ for some $w_1\in \Zb[F]$ and $w_2\in
[0,1)^F$. Then $\tilde{x}=f_Fw_1+f_Fw_2$ and
$\|f_Fw_2\|_2\le \|f\|\cdot \|w_2\|_2\le \|f\|\cdot |F|^{1/2}$. Note that $\tilde{x}$ and
$f_Fw_2$ have the same image in $G$. Thus we may replace
$\tilde{x}$ by $f_Fw_2$ and hence assume that
$\|\tilde{x}\|_2\le \|f\|\cdot |F|^{1/2}$.

Denote by $\varphi$ the quotient map $\Rb[[\Gamma]]\rightarrow
(\Rb/\Zb)[[\Gamma]]$. We identify $\Cb[F]$ with a subspace of $\ell^2(\Gamma)$ naturally.
For any $x\in G_{\rm tor}$, we have
$$ f\varphi(f^{-1}\tilde{x})=\varphi(f(f^{-1}\tilde{x}))=\varphi(\tilde{x})=0$$
in $(\Rb/\Zb)[[\Gamma]]$, and hence $\varphi(f^{-1}\tilde{x})\in X_f$ by \eqref{E-X}.
This defines a map $\psi: G_{\rm
tor}\rightarrow X_f$ sending $x$ to $\varphi(f^{-1}\tilde{x})$.

For each $x\in G_{\rm tor}$, pick $w_x\in
\frac{1}{D}\Zb[KF\setminus F]$ such that
$$\|w_x-p_{KF\setminus
F}(f^{-1}\tilde{x})\|_{\infty}\le 1/D=\varepsilon/2$$
and $|w_x(t)|\le |(f^{-1}\tilde{x})(t)|$ for all $t\in KF\setminus F$.
Then $Dw_x\in \Zb[KF\setminus F]$ and
$$\|Dw_x\|_2\le D\|p_{KF\setminus F}(f^{-1}\tilde{x})\|_2\le
D\cdot \|f^{-1}\|\cdot \|\tilde{x}\|_2 \le D\cdot \|f\|\cdot \|f^{-1}\|\cdot |F|^{1/2}.$$
By
Lemma~\ref{ball:lemma} one has
\begin{eqnarray*}
|\{Dw_x: x\in G_{\rm tor}\}|
&\le & C\lambda^{|F|/2}(D\cdot \|f\|\cdot \|f^{-1}\|)^{|KF\setminus F|}\\
&\le & C\lambda^{|F|/2}(D\cdot \|f\|\cdot \|f^{-1}\|)^{\delta |F|}\\
&\le & C\lambda^{|F|}.
\end{eqnarray*}
Thus
we can find a subset $W\subseteq G_{\rm
tor}$ with $C\lambda^{|F|}|W|\ge |G_{\rm tor}|$ such that
$w_x=w_y$ for all $x, y\in W$.

Now it suffices to show that $\psi$ injects $W$ into an $[F,
\infty, \varepsilon]$-separated subset of $X_f$. Suppose that
$x\neq y$ in $W$ and $d_{F, \infty}(\psi(x), \psi(y))\le
\varepsilon$. From the definition of $d_{F, \infty}$ we have
\begin{align*}
d_{F, \infty}(\psi(x), \psi(y))&=\max_{\gamma \in F}\vartheta((\alpha_f)_\gamma(\psi(x)), (\alpha_f)_\gamma(\psi(y)))\\
&=\max_{\gamma\in F}\vartheta((\psi(x))_\gamma, (\psi(y))_\gamma).
\end{align*}
For each $\gamma\in F$, one gets
$$ \min_{m\in \Zb}|(f^{-1}\tilde{x})_\gamma-(f^{-1}\tilde{y})_\gamma-m|=\vartheta((\psi(x))_\gamma, (\psi(y))_\gamma)\le \varepsilon,$$
and thus there exists $h_\gamma \in \Zb$ with $|(f^{-1}\tilde{x})_\gamma-(f^{-1}\tilde{y})_\gamma-h_\gamma|\le \varepsilon$.
Define $h\in \Zb[F]$ to be the element with value $h_\gamma$ for every $\gamma \in F$.
Set
$$z=f^{-1}\tilde{x}-f^{-1}\tilde{y}-h\in \Rb[[\Gamma]].$$
Then $\|z|_{F}\|_{\infty}\le \varepsilon$.
Since $x$ and $y$ are in
$W$, we have $w_x=w_y$ and hence
\begin{eqnarray*}
\|z|_{KF\setminus F}\|_{\infty}&=&\|p_{KF\setminus
F}(f^{-1}\tilde{x})-p_{KF\setminus
F}(f^{-1}\tilde{y})\|_{\infty} \\
&\le& \|p_{KF\setminus
F}(f^{-1}\tilde{x})-w_x\|_{\infty}+\|p_{KF\setminus
F}(f^{-1}\tilde{y})-w_y\|_{\infty}\le \varepsilon.
\end{eqnarray*}

Write $z$ as $z_1+z_2$ such that the supports of $z_1$ and $z_2$
are contained in $KF$ and $\Gamma\setminus KF$ respectively. Note
that $p_F(fz)=p_F(fz_1)$ and
$\|z_1\|_{\infty}\le \varepsilon$. Consequently,
$$\|p_F(fz)\|_{\infty}=\|p_F(f z_1)\|_{\infty}\le \|f z_1\|_{\infty}\le \|f\|_1\cdot \|z_1\|_{\infty}\le \varepsilon \|f\|_1=1/2.$$
We have
$$ \tilde{x}-\tilde{y}=p_F(\tilde{x}-\tilde{y})=p_F(fh)+p_F(fz)=f_Fh+p_F(fz).$$
Since $\tilde{x}-\tilde{y}$ and $f_Fh$ are both in $\Zb[F]$, we
must have $p_F(fz)=0$. Therefore $\tilde{x}-\tilde{y}=f_Fh\in
f_F\Zb[F]$, contradicting the assumption $x\neq y$. This finishes the proof of
the lemma.
\end{proof}

\begin{theorem} \label{T-positive}
Let $\Gamma$ be an infinite amenable group and let $f\in \Zb\Gamma$ be positive and invertible in $\cL\Gamma$.
Let $\{F_n\}_{n\in J}$ be a (left) F{\o}lner net of $\Gamma$. Then for any $1/(2\|f\|_1)\ge \varepsilon>0$, one has
\begin{eqnarray*}
\rh(\alpha_f)&=&\lim_{n\to \infty}\frac{1}{|F_n|}\log s_{F_n, \infty}(\varepsilon)=\lim_{n\to \infty}\frac{1}{|F_n|}\log |\Zb[F_n]/f_{F_n}\Zb[F_n]|\\
\nonumber &=&\lim_{n\to \infty}\frac{1}{|F_n|}\log |\det f_{F_n}|=\log \ddet_{\cL\Gamma} f.
\end{eqnarray*}
\end{theorem}
\begin{proof} By Theorem~\ref{T-sep} and Lemma~\ref{L-positive small}, one has
$$ \rh(\alpha_f)\le \liminf_{n\to \infty}\frac{1}{|F_n|}\log |\Zb[F_n]/f_{F_n}\Zb[F_n]|.$$
By Lemma~\ref{L-Deninger}, each $f_{F_n}$ is invertible and hence $(\Zb[F_n]/f_{F_n}\Zb[F_n])_{\rm tor}=\Zb[F_n]/f_{F_n}\Zb[F_n]$.
Thus by Theorem~\ref{T-sep} and Lemma~\ref{L-positive large}, one has
$$ \rh(\alpha_f)\ge \limsup_{n\to \infty}\frac{1}{|F_n|}\log s_{F_n, \infty}(\varepsilon)\ge \limsup_{n\to \infty}\frac{1}{|F_n|}\log |\Zb[F_n]/f_{F_n}\Zb[F_n]|,$$
and
$$ \liminf_{n\to \infty}\frac{1}{|F_n|}\log s_{F_n, \infty}(\varepsilon)\ge \liminf_{n\to \infty}\frac{1}{|F_n|}\log |\Zb[F_n]/f_{F_n}\Zb[F_n]|.$$
Then the first two equalities  of the theorem follow. The last two equalities of the theorem come from Lemma~\ref{L-Deninger}.
\end{proof}

\section{Addition Formulas} \label{S-addition}

In this section we establish addition formulas for the entropy of group extensions,
in both topological and measure-theoretical settings (Theorems \ref{group extension1:thm}
and \ref{group extension2:thm}).
From these formulas we deduce
the Yuzvinski\u{\i} addition formula (Corollary~\ref{Yuzvinskii:cor})
and use it to obtain a formula for the entropy of products $fg$ (Corollaries~\ref{C-product} and \ref{zero divisor:cor}).
Throughout this section $\Gamma$ is a countable amenable group.

Let $\alpha_X$, $\alpha_Y$ and $\alpha_G$ be actions of $\Gamma$
on compact metrizable spaces $X$, $Y$ and $G$ by homeomorphisms respectively.
A {\it factor map} $X\rightarrow Y$ is a continuous surjective $\Gamma$-equivariant map.
We say that $\alpha_X$ is a {\it (right) $G$-extension of $\alpha_Y$} if
there are a factor map $\pi: X\rightarrow Y$ and a continuous
map $P: X\times G\rightarrow X$ sending $(x, g)$ to $xg$ such that
$\pi^{-1}(\pi(x))=xG$, $xg=xg'$ only when $g=g'$, and $\gamma(xg)=\gamma(x)\gamma(g)$
for all $x\in X$, $g, g'\in G$ and $\gamma \in \Gamma$.  (Usually $G$ is a compact metrizable group, $(xg)g'=x(gg')$, and
$\Gamma$ acts on $G$ by automorphisms; but this is not necessary.)  The case $\Gamma=\Zb$ of the following
theorem was proved by R. Bowen \cite[Theorem 19]{Bow}.

\begin{theorem}[Topological Addition Formula] \label{group extension1:thm}
Let $\alpha_X$, $\alpha_Y$ and $\alpha_G$ be actions of $\Gamma$
on compact metrizable spaces $X$, $Y$ and $G$ by homeomorphisms respectively.
If $\alpha_X$ is a $G$-extension of $\alpha_Y$, then $\htopol(\alpha_X)=\htopol(\alpha_Y)+\htopol(\alpha_G)$.
\end{theorem}

Let $\alpha_Y$ be an action of $\Gamma$ on
a standard probability space $(Y, \cB_Y, \mu)$ by automorphisms.
Also let $\alpha_G$ be an action of $\Gamma$ on
a compact metrizable group $G$ as (continuous) automorphisms.
Endow $G$ with its Borel $\sigma$-algebra $\cB_G$ and normalized Haar measure $\nu$.
Note that every automorphism of $G$ preserves $\nu$.
A {\it cocycle}  for
$\alpha_Y$ and $\alpha_G$ is a measurable map $\sigma: \Gamma\times Y\rightarrow G$
such that
\begin{eqnarray} \label{cocycle:eq}
\sigma(\gamma_1\gamma_2, y)=\sigma(\gamma_1, \gamma_2y)\cdot \gamma_1(\sigma(\gamma_2, y))
\end{eqnarray}
for all $\gamma_1, \gamma_2\in \Gamma$ and $y\in Y$.
Given a cocycle $\sigma$, one can define a {\it skew product} action $\alpha_Y\times_{\sigma} \alpha_G$ of $\Gamma$ on the standard probability space $(Y\times G, \cB_Y\times \cB_G, \mu \times \nu)$
by automorphisms, by
\begin{eqnarray} \label{extension:eq}
 \gamma(y, g)=(\gamma y, \sigma(\gamma, y)\cdot (\gamma g))
\end{eqnarray}
for $\gamma\in \Gamma$, $y\in Y$ and $g\in G$. It is clear that the projection $Y\times G\rightarrow Y$ is a {\it factor map} for the actions
$\alpha_Y\times_{\sigma}\alpha_G$ and $\alpha_Y$
in the sense that it is $\Gamma$-equivariant, measurable and measure-preserving. The action $\alpha_Y\times_{\sigma}\alpha_G$ is called
a {\it group extension} of the action $\alpha_Y$.  The case $\Gamma=\Zb$ of the following theorem
was proved by Thomas \cite{Thomas2}, and the case $\Gamma=\Zb^d$ for $2\le d<\infty$ was proved by Lind et al. \cite[Theorem B.1]{LSW}.

\begin{theorem}[Measure-theoretical Addition Formula] \label{group extension2:thm}
Let $\alpha_Y$ and $\alpha_G$ be actions of $\Gamma$ on a standard probability space $(Y, \cB_Y, \mu)$ and  a compact metrizable group
$G$ by automorphisms respectively. Let $\sigma$ be a cocycle for $\alpha_Y$ and $\alpha_G$. Then
$$ \rh_{\mu\times \nu}(\alpha_Y\times_{\sigma} \alpha_G)=\rh_{\mu}(\alpha_Y)+\rh(\alpha_G).$$
\end{theorem}

As a direct consequence of Theorem~\ref{group extension1:thm} we obtain
the following Yuzvinski\u{\i} addition formula,
for which the case
$\Gamma=\Zb$ was proved by Yuzvinski\u{\i} \cite{Yuz1} and
the case $\Gamma=\Zb^d$ for $2\le d<\infty$ was proved by Lind et al. \cite[Corollary B.2]{LSW} (see also \cite[Theorem 14.1]{Sch}).
The case $\Gamma=\Zb^{\infty}$ and $G$ is abelian was proved by Miles \cite[Proposition 5.1]{Miles}.
The case $\Gamma$ is locally normal and $G$ is abelian and zero-dimensional was proved by Miles and Bj\"{o}rklund \cite[Theorem 3.1]{MB}.

\begin{corollary}[Yuzvinski\u{\i} Addition Formula] \label{Yuzvinskii:cor}
Let $\alpha_{G_1}$, $\alpha_{G_2}$ and $\alpha_{G_3}$ be actions of $\Gamma$ on
compact metrizable groups $G_1, G_2, G_3$
as (continuous) automorphisms respectively.
Suppose that there is
a $\Gamma$-equivariant short exact sequence of compact groups
$$ 1\longrightarrow G_1\longrightarrow G_2\longrightarrow G_3\longrightarrow 1.$$
Then
$\rh(\alpha_{G_2})=\rh(\alpha_{G_1})+\rh(\alpha_{G_3})$.
\end{corollary}

One can also obtain Corollary~\ref{Yuzvinskii:cor} from Theorem~\ref{group extension2:thm} via a standard procedure, as follows.

\begin{proof}[Proof of Corollary~\ref{Yuzvinskii:cor} using Theorem~\ref{group extension2:thm}]
We may identify $G_1$ with its image in $G_2$.
Denote by $\pi$ the map $G_2\rightarrow G_3$.
Every continuous open surjective map between compact metrizable spaces has a Borel cross section \cite[Theorem 3.4.1]{Arveson}.
Thus we can find a Borel map
$\psi: G_3\rightarrow G_2$ such that $\pi \circ \psi$ is the identity map on $G_3$.
It is easily verified that the map $\phi: G_3\times G_1\rightarrow G_2$ sending $(g_3, g_1)$ to $\psi(g_3)g_1$
is an isomorphism from the measurable space $(G_3\times G_1, \cB_{G_3}\times \cB_{G_1})$ onto
the measurable space $(G_2, \cB_{G_2})$. Furthermore, denoting
the normalized Haar measure on $G_j$ by $\nu_j$, one sees that $\phi(\nu_3\times \nu_1)$ is left-translation invariant
and hence $\phi(\nu_3\times \nu_1)=\nu_2$. It is also readily checked that the map $\sigma:\Gamma\times G_3 \to G_1$
defined by $\sigma(\gamma, g_3)=(\psi(\gamma g_3))^{-1}\cdot \gamma(\psi(g_3))$ is a cocycle for the actions $\alpha_{G_3}$ and $\alpha_{G_1}$,
and that $\phi$ intertwines the actions $\alpha_{G_3}\times_{\sigma}\alpha_{G_1}$ and $\alpha_{G_2}$.
Thus $\rh(\alpha_{G_2})=\rh_{\nu_3\times \nu_1}(\alpha_{G_3}\times_{\sigma} \alpha_{G_1})$.
Theorem~\ref{group extension2:thm} implies that $\rh_{\nu_3\times \nu_1}(\alpha_{G_3}\times_{\sigma} \alpha_{G_1})=\rh_{\nu_3}(\alpha_{G_3})
+\rh(\alpha_{G_1})$. Therefore, $\rh(\alpha_{G_2})=\rh(\alpha_{G_3})+\rh(\alpha_{G_1})$ as desired.
\end{proof}

Now we use Corollary~\ref{Yuzvinskii:cor} to obtain a formula for the entropy of $fg$.
Recall that an element $b$ of a ring $R$ is called a {\it right zero divisor} if
$ab=0$ for some non-zero element $a$ of $R$. The following result was pointed out by Deninger \cite[page 757]{Den}. For the convenience of the reader, we give a proof here.

\begin{lemma} \label{exact:lemma}
Let $f, g\in \Zb\Gamma$. Then one has a $\Gamma$-equivariant short sequence of compact groups
$$ 1\longrightarrow X_g\longrightarrow X_{fg}\longrightarrow
X_f\rightarrow 1,$$
where the homomorphism $ X_{fg}\rightarrow X_f$ is given by left multiplication by $g$.
It is exact at $X_g$ and $X_{fg}$. If furthermore $g$ is not a right zero divisor of $\Zb\Gamma$, then the above sequence
is exact.
\end{lemma}
\begin{proof}
The dual sequence of the above one is the following
\begin{eqnarray} \label{exact:eq}
0\longleftarrow \Zb\Gamma/\Zb\Gamma g\longleftarrow \Zb\Gamma/\Zb\Gamma
fg\longleftarrow \Zb\Gamma/\Zb\Gamma f\longleftarrow
0,
\end{eqnarray}
where the homomorphism $\Zb\Gamma/\Zb\Gamma fg \leftarrow \Zb\Gamma/\Zb\Gamma f$ is given by right multiplication by $g$.
By the Pontryagin duality it suffices to show that
(\ref{exact:eq}) is exact at the corresponding places. Clearly it is exact at
$\Zb\Gamma/\Zb\Gamma g$ and $\Zb\Gamma/\Zb\Gamma fg$. Now assume that $g$ is not a right zero divisor of $\Zb\Gamma$.
Suppose that
$x\in \Zb\Gamma/\Zb\Gamma f$ and $xg=0$ in
$\Zb\Gamma/\Zb\Gamma fg$. Say, $x$ is represented by $\tilde{x}$
in $\Zb\Gamma$. Then $\tilde{x}g=\tilde{z}fg$ in $\Zb\Gamma$
for some $\tilde{z}\in \Zb\Gamma$. Since $g$ is not a right zero divisor in $\Zb\Gamma$,
we have $\tilde{x}=\tilde{z}f$ in $\Zb\Gamma$.
Consequently, $x=0$ and hence (\ref{exact:eq}) is also exact at
$\Zb\Gamma/\Zb\Gamma f$.
\end{proof}

If $\alpha$ is an action of $\Gamma$ on a compact Hausdorff space $X$ by homeomorphisms, and $Y$ is a closed
invariant subspace of $X$, then $\alpha$ restricts to an action $\beta$ of $\Gamma$ on $Y$, and
from the definition of topological entropy one can see easily that $\htopol(\alpha)\ge \htopol(\beta)$.
Combining this fact with Corollary~\ref{Yuzvinskii:cor} and
Lemma~\ref{exact:lemma} we obtain the following product formula.

\begin{corollary} \label{C-product}
Let $f, g\in \Zb\Gamma$. Then $\rh(\alpha_{fg})\le \rh(\alpha_f)+\rh(\alpha_g)$.
If furthermore $g$ is not a right zero divisor in $\Zb\Gamma$, then $\rh(\alpha_{fg})=\rh(\alpha_f)+\rh(\alpha_g)$.
\end{corollary}

The {\it zero divisor conjecture} states that for any torsion-free group $H$, the group ring $\Zb H$ has no nontrivial right zero divisors.
See \cite[page 376--379]{Luck} and \cite[page 62--63]{MV} for relation between the zero divisor conjecture and other conjectures
such as the (strong) Atiyah conjecture and the embedding conjecture. Recall that the class of {\it elementary amenable groups}
is the smallest class of groups containing all cyclic and all finite groups and being closed under taking group extensions and direct unions.
Because of Linnell's work on the strong
Atiyah conjecture \cite{Linnell} (see also \cite{Schick1, DLMSY}), we know that the zero divisor conjecture holds for all torsion-free groups in the smallest class of groups
containing all free groups and being closed under extensions with elementary amenable quotients and under direct unions.
In particular, the zero divisor conjecture holds for all torsion-free elementary amenable groups.
See also \cite[Chapter 13]{Passman} for work on the zero divisor problem of $KH$ for a field $K$ and a group $H$.

If $f=0$ in $\Zb \Gamma$, then $\alpha_f$ is the full shift action of $\Gamma$ on $\Tb^{\Gamma}$ and hence
$\rh(\alpha_f)=\infty$. Thus we have

\begin{corollary} \label{zero divisor:cor}
Suppose that $\Gamma$ is torsion-free and satisfies the zero divisor conjecture. Then for
any $f, g\in \Zb\Gamma$, one has $\rh(\alpha_{fg})=\rh(\alpha_f)+\rh(\alpha_g)$.
\end{corollary}


R. Bowen's proof of Theorems~\ref{group extension1:thm} in the case $\Gamma=\Zb$ is purely topological,
while the proofs of Thomas and Lind et al. for Theorem~\ref{group extension2:thm} in the case
$\Gamma=\Zb^d$ is purely using ergodic theory and depends on a technique of Yuzvinski\u{\i} reducing
$G$ to simpler compact groups. Our proof for these addition formulas, in each setting,
employ both topological and measure-theoretical tools. There are two main tools used in our proof.
One is Ward and Zhang's addition formula \cite[Theorem 4.4]{WZ} (see also \cite[Theorem 0.2]{Dan}), a generalization of the
Abramov-Rohlin addition formula. Another
is the various kinds of fibre entropy
for topological extensions. In particular, our proof of Theorem~\ref{group extension2:thm}, even in the case
$\Gamma=\Zb^d$, is completely different from that of Thomas and Lind et al.

The rest of this section is devoted to the proofs of Theorems~\ref{group extension1:thm} and \ref{group extension2:thm}.
Fix a (left) F{\o}lner sequence $\{F_n\}_{n\in \Nb}$  of $\Gamma$.

A systematic study of various fibre and conditional entropies was carried out in \cite{DS} for
dynamical systems of continuous maps on compact Hausdorff spaces. It will be interesting to see
to what extent the results in \cite{DS} generalize to actions of discrete amenable groups. Here we confine ourselves
to extend a few definitions and results in \cite{DS} to $\Gamma$-actions, needed for the proofs of
Theorems~\ref{group extension1:thm} and \ref{group extension2:thm}.

Let $\alpha_X$ be an action of $\Gamma$ on
a compact metrizable space $X$ by homeomorphisms.
Denote by $M_{\Gamma}(X)$ the set of all $\Gamma$-invariant Borel probability measures on $X$.
For any finite open cover $\cU$ of $X$ and any subset $Z\subseteq
X$, denote by $N(\cU|Z)$ the  minimal number of elements in $\cU$
needed to cover $Z$. Set $\cU^F:=\bigvee_{\gamma \in F} \gamma^{-1} \cU$
for a nonempty finite subset $F$ of $\Gamma$,
and
$$\htopol(\alpha_X, \cU| Z):=\limsup_{n\to \infty} \frac{1}{|F_n|}\log
N(\cU^{F_n}|Z).$$

Let $\alpha_Y$ be an action of $\Gamma$ on
another compact metrizable space $Y$ by homeomorphisms.
Consider a factor map $\pi: X\rightarrow Y$.
Given a  finite open
cover $\cU$ of $X$, note that the function $y\mapsto
N(\cU|\pi^{-1}(y))$ for $y\in Y$ is upper semicontinuous and hence
is a Borel function.
Let $\nu\in \M_{\Gamma}(Y)$.
Set
$$\rH(\cU|\nu):=\int_Y\log
N(\cU|\pi^{-1}(y)) \, d\nu(y).$$
It is easy to verify that the
function $F\mapsto \rH(\cU^F|\nu)$ defined on the set of  nonempty finite subsets
of $\Gamma$ satisfies the hypothesis in Proposition~\ref{subadditive:prop}
 and
hence $\lim_{n\to \infty}\frac{1}{|F_n|}\rH(\cU^{F_n}|\nu)$
exists and does not depend on the choice of the F{\o}lner sequence $\{F_n\}_{n\in \Nb}$.

\begin{definition} \label{fibre:def}
Let $\cU$ be a finite open cover of $X$. For $y\in Y$,
we define {\it the topological fibre entropy of
$\cU$ given $y$} as $\htopol(\alpha_X, \cU|\pi^{-1}(y))$
and denote
it by $\htopol(\alpha_X, \cU|y)$. For any $\nu \in M_{\Gamma}(Y)$,
we define the {\it topological fibre entropy of $\cU$ given $\nu$} as
$\lim_{n\to \infty}\frac{1}{|F_n|}\rH(\cU^{F_n}|\nu)$
and denote it by $\htopol(\alpha, \cU|\nu)$.
We define the {\it topological fibre entropy of $\alpha_X$ given
$y$}, and {\it given $\nu$}, respectively, as $\sup_{\cU}\htopol(\alpha_X, \cU|y)$
and $\sup_{\cU}\htopol(\alpha_X, \cU|\nu)$ respectively for the supremum being
taken over all finite open covers of $X$, and denote them by
$\htopol(\alpha_X|y)$ and $\htopol(\alpha_X|\nu)$ respectively.
\end{definition}

The following result is the analogue of part of \cite[Theorem 3]{DS}.

\begin{lemma} \label{outer CVP:lemma}
Let $\alpha_X$ and $\alpha_Y$ be actions of $\Gamma$ on compact metrizable spaces $X$ and $Y$ respectively.
Let $\pi: X\rightarrow Y$ be a factor map.
Then we have
$$ \sup_{y\in Y}\htopol(\alpha_X|y) \ge \sup_{\nu\in M_{\Gamma}(Y)} \htopol(\alpha_X|\nu).$$
\end{lemma}
\begin{proof}
It suffices to prove $\sup_{y\in Y}\htopol(\alpha_X, \cU|y)\ge
\htopol(\alpha_X, \cU|\nu)$ for every finite open cover $\cU$ of $X$ and every $\nu \in M_{\Gamma}(Y)$. Since the
function $y\mapsto N(\cU^F|\pi^{-1}(y))$ is a Borel function on
$Y$ for any nonempty finite subset $F$ of $\Gamma$, the function
$y\mapsto \htopol(\alpha_X, \cU|y)$ is also Borel. Note that
\begin{align} \label{E-uniform bound}
\frac{1}{|F|}\log N(\cU^{F}|\pi^{-1}(y))\le \frac{1}{|F|}\log N(\cU^{F})\le \log N(\cU)
\end{align}
for any nonempty finite subset $F$ of $\Gamma$ and $y \in Y$.
Thus
\begin{eqnarray*}
 \sup_{y\in Y}\htopol(\alpha_X, \cU|y)&\ge &\int_Y \htopol(\alpha_X, \cU|y)\, d\nu(y) \\
&=& \int_Y \lim_{n\to \infty}\sup_{m\ge n}\frac{1}{|F_m|}\log N(\cU^{F_m}|\pi^{-1}(y))\, d\nu(y) \\
&=& \lim_{n\to \infty}
\int_Y \sup_{m\ge n}\frac{1}{|F_m|}\log N(\cU^{F_m}|\pi^{-1}(y))\, d\nu(y) \\
&\ge &\lim_{n\to \infty} \sup_{m\ge n}
\frac{1}{|F_m|}\int_Y \log N(\cU^{F_m}|\pi^{-1}(y))\, d\nu(y) \\
&=&\lim_{n\to \infty} \sup_{m\ge n} \frac{1}{|F_m|} \rH(\cU^{F_m}|\nu) \\
&=&\htopol(\alpha_X, \cU|\nu),
\end{eqnarray*}
where the third lines comes from Lebesgue's monotone convergence theorem \cite[Theorem 1.26]{Rudin} and the uniform upper bound in \eqref{E-uniform bound}.
\end{proof}

The factor map $\pi: X\rightarrow Y$ induces a surjective continuous affine map from the space $M(X)$ of Borel probability measures
on $X$ to $M(Y)$. For any $\nu \in M_{\Gamma}(Y)$, take $\mu'\in M(X)$ with $\pi(\mu')=\nu$ and let $\mu$ be a limit  point of
the sequence $\{\frac{1}{|F_n|}\sum_{\gamma \in F_n}\gamma \mu'\}_{n\in \Nb}$ in the compact space $M(X)$. Then
$\mu$ is in $M_{\Gamma}(X)$ and $\pi(\mu)=\nu$. Thus
\begin{eqnarray} \label{E-invaraint measure surjective}
\pi(M_{\Gamma}(X))=M_{\Gamma}(Y).
\end{eqnarray}
Note that $\pi^{-1}(\cB_Y)$ is a $\Gamma$-invariant sub-$\sigma$-algebra of $\cB_X$.
We shall identify $\cB_Y$ with $\pi^{-1}(\cB_Y)$, and write $\rH_{\mu}(\cdot|\pi^{-1}(\cB_Y))$ and
$\rh_{\mu}(\cdot |\pi^{-1}(\cB_Y))$ simply
as $\rH_{\mu}(\cdot|\cB_Y)$ and $\rh_{\mu}(\cdot|\cB_Y)$ respectively.

The following result is the analogue of part of \cite[Theorem 4]{DS}.

\begin{lemma} \label{inner VP:lemma}
Let the assumptions be as in Lemma~\ref{outer CVP:lemma}.
For any $\nu\in M_{\Gamma}(Y)$, we have
$$ \htopol(\alpha_X|\nu)\ge \sup_{\mu\in M_{\Gamma}(X), \pi\mu=\nu}\rh_{\mu}(\alpha_X|\cB_Y).$$
\end{lemma}
\begin{proof} We combine the ideas in the proofs of
\cite[Theorem 4]{DS} and \cite[Theorem 5.2.8]{JMO}.
Let $\mu\in M_{\Gamma}(X)$
with $\pi (\mu)=\nu$.
Let
$\cP=\{P_1, \dots, P_k\}$ be a finite Borel partition of $X$ and
let $\varepsilon>0$. It suffices to show that there exists a finite
open cover $\cU$ of $X$ such that $\rh_{\mu}(\alpha_X, \cP|\cB_Y)\le \htopol(\alpha_X, \cU|\nu)+\varepsilon$.

We may assume that $\min_{1\le i\le k}\mu(P_i)>0$. Let $\delta$ be a small positive
constant which we shall determine later. Since
$\mu$ is regular \cite[Theorem 17.11]{K1}, we may find an open set $U_i\supseteq P_i$ for each $1\le i\le k$
such that $\mu(U_i\setminus P_i)<\delta$. Then $\cU=\{U_1, \dots, U_k\}$ is an open
cover of $X$.

Let $F$ be a nonempty finite subset of $\Gamma$. Define an
equivalence relation $\sim$ on $Y$ as $y\sim y'$ whenever
$\pi^{-1}(y)$ and $\pi^{-1}(y')$ are covered by exactly the same
subfamilies of $\cU^F$. Denote by $\beta$ the finite partition of
$Y$ into the equivalence classes. It is readily verified that each
item of $\beta$ is the intersection of a closed set and an open
set, and hence is Borel. For each $D\in \beta$ we can find some
$\cV_D\subseteq \cU^F$ such that $\cV_D$ covers $\pi^{-1}(D)$ and
$|\cV_D|=N(\cU^F|\pi^{-1}(y))$ for every $y\in D$. It is easy to
construct a Borel partition $\cQ_D=\{Q_{D, R}: R\in \cV_D\}$ of
$\pi^{-1}(D)$ with $Q_{D, R}\subseteq R$ for each $R\in \cV_D$.
Set $Q_R:=\bigcup_{D\in \beta}Q_{D, R}$ for $R\in \bigcup_{D\in
\beta}\cV_D$. Then $\cQ:=\{Q_R: R\in \bigcup_{D\in \beta}\cV_D\}$ is
a Borel partition of $X$.
For any finite Borel partition $\cP'$ of $X$,
denote by $\widehat{\cP'}$ the $\sigma$-algebra generated by the items of $\cP'$.
Note that for any $m$-item Borel partition
$\cP'$ of $X$, one has $\rH_{\mu}(\cP')\le \log m$ \cite[page 80]{Wal}.
Thus
\begin{eqnarray} \label{inner1:eq}
\rH_{\mu}(\cQ|\hat{\beta})\le \sum_{D\in \beta}\nu(D)\log |\cV_D|=
\int_Y \log N(\cU^F|\pi^{-1}(y))\, d\nu(y)=\rH(\cU^F|\nu).
\end{eqnarray}

We say that a finite partition $\cP'$ of $X$ is {\it adapted} to a
finite open cover $\cU'$ of $X$ if there is an injective (not necessarily surjective) map $\psi$ from
$\cP'$ to $\cU'$ such that each $P\in \cP'$ is contained in
$\psi(P)$. Denote by $R_{\mu}(\cU')$ the supremum of
$\rH_{\mu}(\cP'|\widehat{\cQ'})$ for all Borel partitions $\cP'$ and $\cQ'$
of $X$
adapted to $\cU'$.
By \cite[Prop. 5.2.11]{JMO} one has
$R_{\mu}(\cU'\vee \cV')\le R_{\mu}(\cU')+R_{\mu}(\cV')$ for all
finite open covers $\cU'$ and $\cV'$ of $X$. Note that both
$\cP^F$ and $\cQ$ are adapted to $\cU^F$ and hence
\begin{eqnarray} \label{inner2:eq}
\rH_{\mu}(\cP^F|\hat{\cQ})\le R_{\mu}(\cU^F)\le |F|R_{\mu}(\cU).
\end{eqnarray}

For two sub-$\sigma$-algebras $\cB_1$ and $\cB_2$ of $\cB_X$,
denote  by $\cB_1\vee \cB_2$ the sub-$\sigma$-algebra of $\cB_X$ generated by $\cB_1$
and $\cB_2$.
We have
\begin{eqnarray*}
\rH_{\mu}(\cP^F|\cB_Y)&\le & \rH_{\mu}(\cP^F\vee \cQ|\cB_Y)\\
&=&  \rH_{\mu}(\cQ|\cB_Y)+\rH_{\mu}(\cP^F|\hat{\cQ}\vee \cB_Y) \\
&\le &  \rH_{\mu}(\cQ|\hat{\beta})+\rH_{\mu}(\cP^F|\hat{\cQ}) \\
&\overset{(\ref{inner1:eq}), (\ref{inner2:eq})}\le & \rH(\cU^F|\nu)+|F|R_{\mu}(\cU).
\end{eqnarray*}
Divide both sides of the above inequality by $|F|$, replace $F$ by $F_n$ and take
limits. We obtain $\rh_{\mu}(\alpha_X, \cP|\cB_Y)\le \htopol(\alpha_X, \cU|\nu)+R_{\mu}(\cU)$. It remains
to show that $R_{\mu}(\cU)\le \varepsilon$ when $\delta$ is small enough.

We may assume that $\delta<\frac{1}{k}\min_{1\le i\le k}\mu(P_i)$. Then the sum of the $\mu$-measures of the
elements in any proper subset of $\cU$ is strictly less than $1$. It follows that
every Borel partition of $X$ adapted to $\cU$ has exactly $k$ items.
Let $\cP'=\{P'_1, \dots, P'_k\}$ and
$\cQ'=\{Q'_1, \dots, Q'_k\}$ be Borel partitions of $X$ adapted to $\cU$ with
$P'_i, Q'_i\subseteq U_i$ for each $1\le i\le k$.
By \cite[Lemma 4.3.9]{JMO} one has
$\rH_{\mu}(\cP'|\widehat{\cQ'})\le 2k^2\xi(2d(\cP', \cQ')/k^2)$, where
$d(\cP', \cQ'):=\frac{1}{2}\sum_{1\le i\le k}\mu(P'_i\bigtriangleup Q'_i)$ and
$\xi(t):=\max_{0\le s\le t} (-s\log s)$ for $0\le t\le 1$.
Note that
\begin{eqnarray*}
\sum_{1\le i\le k}\mu(P'_i\setminus Q'_i) &\le &\sum_{1\le i\le k}\mu(U_i\setminus Q'_i)=\sum_{1\le i\le k}(\mu(U_i)-\mu(Q'_i))\\
&=& \sum_{1\le i\le k}\mu(U_i)-1=\sum_{1\le i\le k}(\mu(U_i)-\mu(P_i))\\
&=&\sum_{1\le i\le k}\mu(U_i\setminus P_i)<k\delta.
\end{eqnarray*}
Similarly, $\sum_{1\le i\le k}\mu(Q'_i\setminus P'_i) <k\delta$. It follows that $d(\cP', \cQ')<k\delta$.
Thus $R_{\mu}(\cU)\le 2k^2\xi(2\delta/k)$. Therefore it suffices
to require further $\xi(2\delta/k)\le \varepsilon/(2k^2)$.
\end{proof}


The case $\Gamma=\Zb$ of the next theorem was proved by R. Bowen \cite[Theorem 17]{Bow}.
Our proof for the general case takes the approach in \cite{DS}.

\begin{theorem} \label{fibre:thm}
Let the assumptions be as in Lemma~\ref{outer CVP:lemma}.
We have
$$ \htopol(\alpha_X)\le \htopol(\alpha_Y)+\sup_{y\in Y}\htopol(\alpha_X|y).$$
\end{theorem}
\begin{proof} By Theorem 0.2 of \cite{Dan}, when $\Gamma$ is infinite, we have
\begin{eqnarray} \label{ZW addition:eq}
\rh_{\mu}(\alpha_X)=\rh_{\pi \mu}(\alpha_Y)+\rh_{\mu}(\alpha_X|\cB_Y)
\end{eqnarray}
for every $\mu \in M_\Gamma(X)$. If $\Gamma$ is finite and $\alpha_Z$ is an action of $\Gamma$ on a standard probability space $(Z, \cB_Z, \mu_Z)$ by automorphisms and $\cD$ is a $\Gamma$-invariant sub-$\sigma$-algebra of $\cB_Z$,
then clearly $\rh_{\mu_Z}(\alpha_Z|\cD)=\frac{H(\mu_Z|\cD)}{|\Gamma|}$,  where $H(\mu_Z|\cD)$ denotes the supremum of $H_{\mu_Z}(\cP|\cD)$ for $\cP$ running over
all finite measurable partitions of $Z$. If $\mu_Z$ is purely atomic in the sense that $\sum_{z\in Z}\mu_Z(\{z\})=1$, then
$H(\mu_Z|\{\emptyset, Z\})=\sum_{z\in Z}-\mu_Z(\{z\})\log \mu_Z(\{z\})$. If $\mu_Z$ is not purely atomic, then there is some
$Z'\in \cB_Z$ with $\mu_Z(Z')>0$ such that $Z'$ equipped with the restriction of $\cB_Z$ and $\mu_Z$ is isomorphic to the interval $[0, \mu_Z(Z')]$ equipped with the Borel structure of
its canonical topology and the Lebesgue measure \cite[Theorem 17.41]{K1}, and hence $H(\mu_Z|\{\emptyset, Z\})=\infty$.
It follows easily that the formula (\ref{ZW addition:eq}) holds also when $\Gamma$ is finite.

By the variational principle \cite[page 76]{JMO} we have
$\htopol(\alpha_X)=\sup_{\mu\in M_{\Gamma}(X)} \rh_{\mu}(\alpha_X)$ and
$\htopol(\alpha_Y)=\sup_{\nu\in M_{\Gamma}(Y)} \rh_{\nu}(\alpha_Y)$.
Thus Theorem~\ref{fibre:thm} follows from Lemmas~\ref{outer CVP:lemma}
and \ref{inner VP:lemma}, and (\ref{E-invaraint measure surjective}).
\end{proof}

Fix a compatible matric $d$ on $X$. For any $\varepsilon>0$ and
any nonempty finite subset $F\subseteq \Gamma$, we say that a set
$E\subseteq X$ is {\it $(F, \varepsilon)$-separated} if for any
$x\neq y$ in $E$ there is some $\gamma \in F$ with $d(\gamma
x, \gamma y)> \varepsilon$ and we say that a set $E'\subseteq
X$ {\it $(F, \varepsilon)$-spans} another subset $Z\subseteq X$ if
for any $x\in Z$ there is some $y\in E'$ with $d(\gamma x,
\gamma y)\le \varepsilon$ for all $\gamma \in F$. For any
$Z\subseteq X$, denote by $r_F(\varepsilon, Z)$ the smallest
cardinality of any set $E$ which $(F, \varepsilon)$-spans $Z$ and
denote by $s_F(\varepsilon, Z)$ the largest cardinality of any
$(F, \varepsilon)$-separated set $E$ contained in $Z$.

It is routine to prove the following lemma (cf. \cite[Lemma 1]{Bow}
\cite[Prop 2.1]{Den}).

\begin{lemma} \label{Bowen:lemma}
Let $\alpha_X$ be an action of $\Gamma$ on a compact metrizable space $X$ by homeomorphisms.
For any $Z\subseteq X$, we have
$$ \sup_{\cU}\htopol(\alpha_X, \cU|Z)=\lim_{\varepsilon \to 0}\limsup_{n\to \infty}
\frac{1}{|F_n|}\log r_{F_n}(\varepsilon, Z)=
\lim_{\varepsilon \to 0}\limsup_{n\to \infty}
\frac{1}{|F_n|}\log s_{F_n}(\varepsilon, Z),$$
where the supremum is taken  over all finite open covers of $X$.
\end{lemma}

We are ready to prove Theorem~\ref{group extension1:thm}.

\begin{proof}[Proof of Theorem~\ref{group extension1:thm}]
If $\Gamma$ is finite and $\alpha_Z$ is an action of $\Gamma$ on a compact Hausdorff space $Z$
by homeomorphisms, then clearly $\htopol(\alpha_Z)=\frac{\log |Z|}{|\Gamma|}$ when $Z$ is finite while
$\htopol(\alpha_Z)=\infty$ when $Z$ is infinite. It follows that Theorem~\ref{group extension1:thm} holds
when $\Gamma$ is finite. Thus we may assume that $\Gamma$ is infinite.
We follow the proof of
\cite[Theorem 19]{Bow}, but using Theorem~\ref{fibre:thm} and Lemma~\ref{Bowen:lemma}.

Fix compatible metrics $d_X$, $d_Y$ and $d_G$ for $X$, $Y$ and $G$ respectively.
To show $\htopol(\alpha_X)\le \htopol(\alpha_Y)+\htopol(\alpha_G)$, by Theorem~\ref{fibre:thm} it suffices to
show $\htopol(\alpha_X|y)\le \htopol(\alpha_G)$ for every $y\in Y$. Take $z\in \pi^{-1}(y)$. Given $\varepsilon>0$,
take $\delta>0$ such that $d_X(xg_1, xg_2)\le \varepsilon$ for any $x\in X$ and $g_1, g_2\in G$ with
$d_G(g_1, g_2)\le \delta$. Let $F$ be a nonempty finite subset of $\Gamma$. If a subset $E$ of $G$ $(F, \delta)$-spans
$G$, then $zE$ $(F, \varepsilon)$-spans $zG=\pi^{-1}(y)$. Thus $r_F(\varepsilon, \pi^{-1}(y))\le r_F(\delta, G)$.
By Lemma~\ref{Bowen:lemma} we get $\htopol(\alpha_X|y)\le \htopol(\alpha_G)$ as desired.

Next we show $\htopol(\alpha_X)\ge \htopol(\alpha_Y)+\htopol(\alpha_G)$. Given $\varepsilon>0$, since $X$ and $G$ are compact
and $xg=xg'$ only when $g=g'$, we can
find $\delta>0$ such that $d_X(x_1, x_2)>\delta$ for any $x_1, x_2\in X$ with
$d_Y(\pi(x_1), \pi(x_2))> \varepsilon$, and that $d_X(xg_1, xg_2)>\delta$ for any $x\in X$ and $g_1, g_2\in G$ with
$d_G(g_1, g_2)>\varepsilon$. Let $F$ be a nonempty finite subset of $\Gamma$. Let $E_Y$ and $E_G$ be subsets
of $Y$ and $G$ being $(F, \delta)$-separated respectively. Take $E_X\subseteq X$ such that the restriction
of $\pi$ on $E_X$ maps $E_X$ bijectively to $E_Y$. We claim that $|E_XE_G|=|E_X|\cdot |E_G|$ and that
$E_XE_G$ is $(F, \delta)$-separated.
If $x_1, x_2$ are distinct points in $E_X$ and $g_1, g_2\in E_G$, then $\pi(x_1), \pi(x_2)\in E_Y$ are distinct,
thus for some $\gamma\in F$
one has $d_Y(\pi(\gamma (x_1g_1)), \pi(\gamma (x_2g_2)))=
d_Y(\gamma \pi(x_1), \gamma \pi(x_2))> \varepsilon$ and hence $d_X(\gamma (x_1g_1), \gamma (x_2g_2))>\delta$.
If $g_1, g_2$ are distinct points in $E_G$ and $x\in E_X$,
then
for some $\gamma\in F$
one has $d_G(\gamma (g_1), \gamma (g_2))> \varepsilon$ and hence $d_X(\gamma(xg_1), \gamma (xg_2))=d_X(\gamma (x)\gamma (g_1), \gamma (x)\gamma (g_2))>\delta$. This proves the claim. Thus $s_F(\delta, X)\ge s_F(\varepsilon, Y)s_F(\varepsilon, G)$.
By Lemma~\ref{Bowen:lemma} we get $\htopol(\alpha_X)\ge \htopol(\alpha_Y)+\htopol(\alpha_G)$ as desired.
\end{proof}

Let $X$ be a $G$-extension of $Y$. In the second paragraph of the proof of Theorem~\ref{group extension1:thm},
we have proved that $\htopol(\alpha_X|y)\le \htopol(\alpha_G)$ for every $y\in Y$. The argument in the third paragraph of the proof
also shows that $\htopol(\alpha_X|y)\ge \htopol(\alpha_G)$ for every $y\in Y$. For later use, we record this as

\begin{lemma} \label{fibre:lemma}
Let the assumptions be as in Theorem~\ref{group extension1:thm}.
Then $\htopol(\alpha_X|y)=\htopol(\alpha_G)$ for every $y\in Y$.
\end{lemma}

Next we consider group extensions constructed out of continuous cocycles.

\begin{lemma} \label{conditional:lemma}
Let $\alpha_Y$ and $\alpha_G$ be actions of $\Gamma$ on
a compact metrizable space $Y$ and a compact metrizable group $G$ by homeomorphisms
and (continuous) automorphisms respectively. Let $\sigma: \Gamma \times Y\rightarrow G$ be a continuous cocycle, i.e.
a continuous map satisfying (\ref{cocycle:eq}). Consider the action $\alpha_Y\times_{\sigma}\alpha_G$
of $\Gamma$ on the compact metrizable space $Y\times G$ by homeomorphisms, defined
by (\ref{extension:eq}). For any $\mu\in M_{\Gamma}(Y)$, denoting by $\nu$ the normalized Haar measure of $G$, we have
$$ \rh_{\mu\times \nu}(\alpha_Y\times_{\sigma}\alpha_G|\cB_Y)= \rh(\alpha_G).$$
\end{lemma}
\begin{proof}
Note that $\alpha_Y\times_{\sigma} \alpha_G$ is a $G$-extension of $\alpha_Y$ and $\pi(\mu\times \nu)=\mu$, where $\pi$ denotes the projection $Y\times G\rightarrow Y$.
From Lemmas~\ref{outer CVP:lemma}, \ref{inner VP:lemma} and \ref{fibre:lemma} we
have
$$\rh_{\mu\times \nu}(\alpha_Y\times_{\sigma} \alpha_G|\cB_Y)\le \htopol(\alpha_Y\times_{\sigma} \alpha_G|\mu)\le \sup_{y\in Y}\htopol(\alpha_Y\times_{\sigma} \alpha_G|y)=\rh(\alpha_G).$$
Thus it suffices to show $ \rh_{\mu\times \nu}(\alpha_Y\times_{\sigma} \alpha_G|\cB_Y)\ge \rh(\alpha_G)$.

Take compatible metrics $d_Y$ and $d_G$ on $Y$ and $G$ respectively. Replacing $d_G(\cdot, \cdot)$
by $\int_G d_G(g\cdot, g\cdot)\, d\nu(g)$ if necessary, we may assume that $d_G$ is left-translation invariant.
We endow $Y\times G$ with the metric $d_{Y\times G}((y_1, g_1), (y_2, g_2))=\max(d_Y(y_1, y_2), d_G(g_1, g_2))$.

Let $\varepsilon>0$ and $F$ be a nonempty finite subset of $\Gamma$.
Let $E$ be an $(F, \varepsilon)$-separated subset of $G$ with $|E|=s_F(\varepsilon, G)$.
Set $V=\{g\in G: \max_{\gamma \in F}d_G(\gamma g, e_G)\le \varepsilon/2\}$, where
$e_G$ denotes the identity element of $G$.
Then $V$ is a closed subset of $G$, and the sets $gV$ for $g\in E$ are pairwise disjoint.
Thus $1\ge \nu(\bigcup_{g\in E}gV)=\sum_{g\in E}\nu(gV)=|E|\nu(V)$. Therefore $\nu(V)\le |E|^{-1}$.

Let $\cP$ be a finite Borel partition of $Y\times G$ with each item having diameter no bigger than $\varepsilon/2$, under $d_{Y\times G}$.
Let $P$ be an item of $\cP^F$, and let $(y, g_1), (y, g_2)\in P$. Then
for each $\gamma\in F$, one has
\begin{eqnarray*}
\varepsilon/2 &\ge &d_{Y\times G}(\gamma (y, g_1), \gamma (y, g_2))\\
&=&d_{Y\times G}((\gamma y, \sigma(\gamma, y)(\gamma g_1)),
  (\gamma y, \sigma(\gamma, y)(\gamma g_2)))\\
  &=&d_G(\sigma(\gamma, y)(\gamma g_1), \sigma(\gamma, y)(\gamma g_2))\\
  &=&
  d_G(\gamma g_1, \gamma g_2))=d_G(\gamma (g_1^{-1}g_2), e_G),
\end{eqnarray*}
where the last two equalities come from the left-translation invariance of $d_G$. Thus
$g_1^{-1}g_2\in V$ and hence $g_2\in g_1V$. It follows that
$$ \Eb(1_{P}|\cB_Y)(x)=\int_G1_P(\pi(x), g')\, d\nu(g')\le \nu(V)\le |E|^{-1}$$
for $\mu \times \nu$ a.e. $x\in Y\times G$, where $1_P$ denotes the characteristic function of $P$. Therefore
\begin{eqnarray*}
\rH_{\mu\times \nu}(\cP^F|\cB_Y)
&=&\sum_{P\in \cP^F}\int_{Y\times G}-1_P(x)\log \Eb(1_P|\cB_Y)(x)\, d(\mu\times \nu)(x) \\
&\ge & \sum_{P\in \cP^F}\int_{Y\times G}-1_P(x)\log |E|^{-1}\, d(\mu\times \nu)(x) =\log |E|=\log s_F(\varepsilon, G).
\end{eqnarray*}
It follows that $\rh_{\mu\times \nu}(\alpha_X\times_{\sigma}\alpha_G, \cP|\cB_Y)\ge \limsup_{n\to \infty}\frac{1}{|F_n|}\log s_{F_n}(\varepsilon, G)$.
By Lemma~\ref{Bowen:lemma} we get $\rh_{\mu\times  \nu}(\alpha_Y\times_{\sigma} \alpha_G|\cB_Y)\ge \rh(\alpha_G)$ as desired.
\end{proof}

Now we show that every measure-theoretical group extension has a topological model.

\begin{lemma} \label{model:lemma}
Let the assumptions be as in Theorem~\ref{group extension2:thm}.
Then there exists a compact metrizable space $Y'$ containing $Y$ such that
$Y$ is a dense Borel subset of $Y'$, $\cB_Y$ is the restriction of $\cB_{Y'}$ on $Y$, the action of $\Gamma$ on $Y$ extends
to an action of $\Gamma$ on $Y'$ by homeomorphisms, the measure $\mu$ extends to a $\Gamma$-invariant
Borel probability measure on $Y'$, and $\sigma$ extends to a continuous cocycle $\Gamma\times Y'\rightarrow G$.
\end{lemma}
\begin{proof}
Denote by $\sB(Y)$ the set of bounded $\Cb$-valued Borel functions on $Y$.
It is complete under the supremum norm $\|\cdot \|$, and is a unital algebra under the pointwise addition
and multiplication. Furthermore, it is a $*$-algebra with the $*$-operation defined
by $f^*(y)=\overline{f(y)}$ for $f\in \sB(Y)$ and $y\in Y$. It is clear that
$\| f^*f\|=\| f\|^2$ for every $f\in \sB(Y)$. Thus $\sB(Y)$ is a unital commutative
$C^*$-algebra (see Section~\ref{SS-determinant and algebra}). Note that the action of $\Gamma$ on $Y$ induces an action of $\Gamma$ on $\sB(Y)$
as isometric $*$-algebra automorphisms naturally.

Since $\cB_Y$ is the Borel $\sigma$-algebra for some Polish topology on $Y$,
we can find  a countable subset $W$ of $\cB_Y$ separating the points of $Y$.
That is, for any distinct $y_1, y_2$ in $Y$, we can find $A\in W$ such that
$1_A(y_1)\neq 1_A(y_2)$, where $1_A$ denotes the characteristic function of $A$.
Set $V_1=\{1_A\in \sB(Y): A\in W\}$.

Note that the algebra $C(G)$ of continuous $\Cb$-valued functions on $G$ is also a normed space under the
supremum norm.
Since $G$ is compact metrizable, $C(G)$ is separable.
Write $\sigma$ as $\sigma_{\gamma}:Y\rightarrow G$ for $\gamma\in \Gamma$. That is,
$\sigma_{\gamma}(y)=\sigma(\gamma, y)$ for $\gamma\in \Gamma$ and $y\in Y$. Then $f\circ \sigma_{\gamma}$ is
in $\sB(Y)$ for every $f\in C(G)$ and $\gamma \in \Gamma$. Set $V_2=\{f \circ \sigma_{\gamma}\in \sB(Y): f\in C(G), \gamma\in \Gamma\}$.
Since $C(G)$ is separable and $\Gamma$ is countable, $V_2$ is a separable subset of $\sB(Y)$.

Denote by $\sA$ the closed $\Gamma$-invariant sub-$*$-algebra of $\sB(Y)$ generated by $V_1\cup V_2$.
Then $\sA$ is separable and contains the constant functions. Denote by $Y'$ the {\it Gelfand spectrum} of
$\sA$, i.e., the set of all unital algebra homomorphisms $\sA\rightarrow \Cb$ \cite[page 219]{Conway}. Note
that $Y'$ is contained in the unit ball of the Banach space dual $\sA'$ of $\sA$ \cite[Proposition VII.8.4]{Conway}.
Endowed with the relative weak$^*$-topology, $Y'$ is a compact Hausdorff space \cite[Proposition VII.8.6]{Conway}.
Since $\sA$ is separable, $Y'$ is metrizable. Clearly the action of $\Gamma$ on $\sA$ induces an action
of $\Gamma$ on $Y'$ by homeomorphisms.

For each $y\in Y$, the evaluation at $y$ gives rise to an element $\psi(y)$ of $Y'$.
Since $W$ separates the points of $Y$, the map $\psi: Y\rightarrow Y'$ is injective.
Consider the {\it Gelfand transform} $\varphi: \sA\rightarrow C(Y')$ defined by
$\varphi(a)(y')=y'(a)$ for $a\in \sA$ and $y'\in Y'$ \cite[page 220]{Conway}.
Note that $\sA$ is a unital commutative $C^*$-algebra.
Thus $\varphi$ is an isometric $*$-isomorphism of $\sA$ onto $C(Y')$ \cite[Theorem VIII.2.1]{Conway}.
Also note that $\varphi(f)\circ \psi=f$ for
every $f\in \sA$. It follows that $\psi$ is measurable and $\Gamma$-equivariant. Recall that a measurable
space $(X, \cB_X)$ is called a  {\it standard Borel space} if
$\cB_X$ is the Borel $\sigma$-algebra for some Polish topology on $X$. The Lusin-Souslin theorem
says that for any injective measurable map $\zeta$ from one standard Borel space $(X, \cB_X)$ to another
standard Borel space $(Z, \cB_Z)$,
the image $\zeta(X)$ is measurable and $\zeta$ is an isomorphism from $(X, \cB_X)$ to $(\zeta(X), \cB_Z|_{\zeta(X)})$ \cite[page 89]{K1}.
Thus, identifying $Y$ with $\psi(Y)$, we have $Y\in \cB_{Y'}$ and $\cB_Y$ is the restriction of $\cB_{Y'}$ on $Y$.
Then $\mu$ can be though of as a Borel probability measure on $Y'$ via setting $\mu(Y'\setminus Y)=0$. Clearly
$\mu$ is still $\Gamma$-invariant. Note that $Y$ separates $\varphi(\sA)=C(Y')$.
By the Urysohn lemma \cite[page 115]{Kelley}, for any disjoint nonempty closed subsets $Z_1$ and $Z_2$ of $Y'$, there exists
$f\in C(Y')$ with $f|_{Z_1}=1$ and $f|_{Z_2}=0$. It follows that $Y$ is dense in $Y'$.

Each $\gamma\in \Gamma$ and each $y'\in Y'$
give rise to a unital algebra homomorphism $C(G)\rightarrow \Cb$ sending $f$ to $y'(f\circ \sigma_{\gamma})$.
Note that every unital algebra homomorphism $C(G)\rightarrow \Cb$ is given by the evaluation at a unique point
of $G$ \cite[Theorem VII.8.7]{K1}. Thus there is a unique point in $G$, denoted by $\sigma'_{\gamma}(y')$, such that
$f(\sigma'_{\gamma}(y'))= y'(f\circ \sigma_{\gamma})$ for every $f\in C(G)$. Clearly the map $\sigma'_{\gamma}:Y'\rightarrow G$
is continuous and extends $\sigma_{\gamma}$ for every $\gamma\in \Gamma$. Write $\sigma'(\gamma, y')$ for
$\sigma'_{\gamma}(y')$. Since $Y$ is dense in $Y'$, by continuity $\sigma'$ also satisfies the cocycle condition (\ref{cocycle:eq}).
\end{proof}

We are ready to prove Theorem~\ref{group extension2:thm}.

\begin{proof}[Proof of Theorem~\ref{group extension2:thm}]
By Lemmas~\ref{model:lemma} and \ref{conditional:lemma} we have $\rh_{\mu\times \nu}(\alpha_Y\times_{\sigma} \alpha_G|\cB_Y)= \rh(G)$.
Thus the desired formula follows from
Ward and Zhang's addition formula $\rh_{\mu\times \nu}(\alpha_Y\times_{\sigma} \alpha_G)=\rh_{\mu}(\alpha_Y)+\rh_{\mu\times \nu}(\alpha_Y\times_{\sigma} \alpha_G|\cB_Y)$ \cite[Theorem 4.4]{WZ} \cite[Theorem 0.2]{Dan}.
\end{proof}

\section{Approximation of Fuglede-Kadison Determinant} \label{S-determinant}

Throughout this section $\Gamma$ will be a discrete amenable group.
As we pointed out at the end of Section~\ref{S-finite group}, one of the main difficulties to establish the intuitive equalities
(\ref{E-intuition}) is that $f_{F_n}$ may fail to be invertible even when $f$ is invertible in $\cL\Gamma$. Our method of dealing with this difficulty is to ``perturb" $f_{F_n}$ to make it invertible. Here the meaning of $S_n\in B(\Cb[F_n])$ being a perturbation of $f_{F_n}$ is that $\rank(S_n-f_{F_n})$ is small compared to $|F_n|$. Our task in the section is to calculate $\ddet_{\cL\Gamma}f$ in terms of the determinants of $S_n$. Though Corollary~\ref{C-determinant formula} gives a precise formula for such a calculation, for some technical reason which will be explained in Remark~\ref{R-uniform bound}, we have
to get an approximate formula as follows:

\begin{theorem} \label{T-approximate determinant formula}
Let $f\in \Cb\Gamma$ be  invertible in $\cL\Gamma$. For any $C_1>0$ and $\varepsilon>0$, there exists $\delta>0$ such
that if $\{F_n\}_{n\in J}$ is  a (left) F{\o}lner net of $\Gamma$ and  $S_n\in B(\Cb[F_n])$  is  invertible for each
$n\in J$ such that $\sup_{n\in J}\max(\|S_n\|, \|S_n^{-1}\|)\le C_1$ and $\limsup_{n\to \infty}\frac{\rank(S_n-f_{F_n})}{|F_n|}\le \delta$, then
\begin{eqnarray*}
\limsup_{n\to \infty}\big|\log \ddet_{\cL\Gamma} f-\frac{1}{|F_n|}\log |\det S_n|\big|<\varepsilon.
\end{eqnarray*}
\end{theorem}
\begin{proof} Let $\delta>0$ be a small number whose value will be determined later. Let $ \{F_n\}_{n\in J}$ and $\{S_n\}_{n\in J}$ satisfy
the hypothesis.

Note that $\sup_{n\in J}\max(\|S^*_nS_n\|, \|(S^*_nS_n)^{-1}\|)\le C^2_1$.
Thus there is a closed finite interval $I$ in $\Rb$ depending only on $C_1$ such that $I$ does not contain $0$ and the spectra of
$f^*f$ and
$S^*_nS_n$ are contained in $I$ for each $n\in J$.

Let $n\in J$.  From
\begin{align*}
S_n^*S_n-(f_{F_n})^*f_{F_n}=(S_n-f_{F_n})^*S_n+(f_{F_n})^*(S_n-f_{F_n})
\end{align*}
we have
\begin{align*}
\rank(S_n^*S_n-(f_{F_n})^*f_{F_n})&\le \rank((S_n-f_{F_n})^*S_n)+\rank((f_{F_n})^*(S_n-f_{F_n}))\\
&\le \rank((S_n-f_{F_n})^*)+\rank(S_n-f_{F_n})\\
&= 2\rank(S_n-f_{F_n}).
\end{align*}
Recall the operators $p_{F_n}$ and $\iota_{F_n}$ in Notation~\ref{N-compression} and the set $K_f$ in Notation~\ref{N-support}. When restricted on $\ell^2(\Gamma)$, one has $p_{F_n}=(\iota_{F_n})^*$.
Thus
\begin{align*}
(f_{F_n})^*f_{F_n}-(f^*f)_{F_n}&=(f_{F_n})^*f_{F_n}-p_{F_n}f^*f\iota_{F_n} \\
&=(f_{F_n})^*f_{F_n}-(f\iota_{F_n})^*f\iota_{F_n} \\
&=(f_{F_n}-f\iota_{F_n})^*f_{F_n}+(f\iota_{F_n})^*(f_{F_n}-f\iota_{F_n}),
\end{align*}
and hence
\begin{align*}
\rank((f_{F_n})^*f_{F_n}-(f^*f)_{F_n})&\le \rank((f_{F_n}-f\iota_{F_n})^*f_{F_n})+\rank((f\iota_{F_n})^*(f_{F_n}-f\iota_{F_n}))\\
&\le \rank((f_{F_n}-f\iota_{F_n})^*)+\rank(f_{F_n}-f\iota_{F_n})\\
&=2\rank(f_{F_n}-f\iota_{F_n})\\
&\le 2|K_fF_n\setminus F_n|.
\end{align*}
Therefore
\begin{align*}
 \rank(S^*_nS_n-(f^*f)_{F_n})&\le \rank(S_n^*S_n-(f_{F_n})^*f_{F_n})+\rank((f_{F_n})^*f_{F_n}-(f^*f)_{F_n})\\
&\le 2\rank(S_n-f_{F_n})+2|K_fF_n\setminus F_n|.
\end{align*}
It follows that
$$ \limsup_{n\to \infty}\frac{\rank(S^*_nS_n-(f^*f)_{F_n})}{|F_n|}\le \limsup_{n\to \infty}\frac{2\rank(S_n-f_{F_n})}{|F_n|}\le 2\delta.$$

Denote by $\tr$ the trace of $B(\Cb[F_n])$ taking value $1$ on minimal projections.
By the Weierstrass approximation theorem \cite[page 312]{Rudin} we can find a real polynomial $Q$ such that $|Q(x)-\log x|\le \varepsilon/2$ for all $x\in I$. Then
\begin{eqnarray} \label{E-approximate1}
\frac{1}{|F_n|}|\tr(Q(S))-\tr(\log S)|\le \|Q(S)-\log S\|\le \varepsilon/2
\end{eqnarray}
for all self-adjoint $S\in B(\Cb[F_n])$ with spectrum contained in $I$, and
\begin{eqnarray} \label{E-approximate2}
\|\tr_{\cL\Gamma}(Q(T))-\tr_{\cL\Gamma}(\log T)\|\le \|Q(T)-\log T\|\le \varepsilon/2
\end{eqnarray}
for all self-adjoint $T\in \cL\Gamma$ with spectrum contained in $I$.

For noncommutative variables $X$ and $Y$, we have
$Q(X+Y)=Q(X)+\sum_{j=1}^kQ_j(X, Y)$ for some two-variable noncommutative monomials $Q_j$ with $Y$ appearing in $Q_j$.
Fix $1\le j\le k$. Then $\sup_{n\in J}\|Q_j((f^*f)_{F_n}, S^*_nS_n-(f^*f)_{F_n})\|\le D_j$ for some constant $D_j$ depending only on $Q_j$, $\|f\|$ and $C_1$. Furthermore,
$$\limsup_{n\to \infty}\frac{\rank(Q_j((f^*f)_{F_n}, S^*_nS_n-(f^*f)_{F_n}))}{|F_n|}\le \limsup_{n\to \infty}\frac{\rank(S^*_nS_n-(f^*f)_{F_n})}{|F_n|}\le  2\delta.$$
For any $S\in B(\Cb[F_n])$, extending an orthonormal basis $e_1, \dots, e_{\rank(S)}$ of the range of $S$ to
an orthonormal basis $e_1, \dots, e_{|F_n|}$ of $\Cb[F_n]$, one sees that
$e_{\rank(S)+1},\dots, e_{|F_n|}$ are orthogonal to the range of $S$, and hence
$$|\tr(S)|=|\sum_{j=1}^{|F_n|}\left< Se_j, e_j\right>|=|\sum_{j=1}^{\rank(S)}\left< Se_j, e_j\right>|\le \sum_{j=1}^{\rank(S)}|\left< Se_j, e_j\right>|\le \rank(S)\cdot \|S\|.$$
It follows that
$\limsup_{n\to \infty}\frac{|\tr(Q_j((f^*f)_{F_n}, S^*_nS_n-(f^*f)_{F_n}))|}{|F_n|}\le 2\delta D_j$, and hence
$$ \limsup_{n\to \infty}\frac{1}{|F_n|}|\tr(Q(S^*_nS_n))-\tr(Q((f^*f)_{F_n}))|\le 2\delta D,$$
where $D=\sum_{j=1}^kD_j$.
By a result of L\"{u}ck and Schick \cite{Luck94} \cite[Lemma 4.6]{Schick01} \cite[Lemma 13.42]{Luck} \cite[page 745]{Den}, for any $T\in \cL\Gamma$ one has
$$ \tr_{\cL\Gamma}(Q(T))=\lim_{n\to \infty}\frac{1}{|F_n|}\tr(Q(T_{F_n})).$$
Thus
\begin{eqnarray} \label{E-approximate3}
 \limsup_{n\to \infty}\big|\tr_{\cL\Gamma}(Q(f^*f))-\frac{1}{|F_n|}\tr(Q(S^*_nS_n))\big|\le 2\delta D.
\end{eqnarray}

Combining (\ref{E-approximate1}), (\ref{E-approximate2}) and (\ref{E-approximate3}) together, we get
$$ \limsup_{n\to \infty}\big|\tr_{\cL\Gamma}(\log (f^*f))-\frac{1}{|F_n|}\tr(\log (S^*_nS_n))\big|\le \varepsilon+2\delta D.$$
That is,
$$ \limsup_{n\to \infty}\big|\log \ddet_{\cL\Gamma} (f^*f)-\frac{1}{|F_n|}\log \det (S^*_nS_n)\big|\le \varepsilon+2\delta D.$$
As $\log \ddet_{\cL\Gamma} (f^*f)=2\log \ddet_{\cL\Gamma} f$ and $\det(S^*_nS_n)=|\det S_n|^2$, we get
\begin{eqnarray*}
\limsup_{n\to \infty}\big|\log \ddet_{\cL\Gamma} f-\frac{1}{|F_n|}\log |\det S_n|\big|\le \varepsilon/2+\delta D.
\end{eqnarray*}
Now we just need to take $\delta<\varepsilon/(2D)$.
\end{proof}

\begin{corollary} \label{C-determinant formula}
Let $f\in \Cb\Gamma$ be  invertible in $\cL\Gamma$. Let $\{F_n\}_{n\in J}$ be a (left) F{\o}lner net of $\Gamma$ and  $S_n\in B(\Cb[F_n])$  be   invertible for each $n\in J$
such that $\sup_{n\in J}\max(\|S_n\|, \|S_n^{-1}\|)<\infty$ and $\lim_{n\to \infty}\rank(S_n-f_{F_n})/|F_n|=0$. Then
\begin{eqnarray*}
\log \ddet_{\cL\Gamma} f=\lim_{n\to \infty}\frac{1}{|F_n|}\log |\det S_n|.
\end{eqnarray*}
\end{corollary}

\section{Proof of $\rh(\alpha_f)\ge \log \ddet_{\cL\Gamma}f$} \label{S-proof of ineqaulity}

In this section we show $\rh(\alpha_f)\ge \log \ddet_{\cL\Gamma}f$ for any $f\in \Zb\Gamma$ invertible in $\cL\Gamma$ (Lemma~\ref{L-lower bound}).
Throughout this section $\Gamma$ is a discrete amenable group.

For $f\in \Cb\Gamma$, recall that $K_f$ denotes the union of the supports of $f$ and $f^*$, and the identity of $\Gamma$.
For a finite subset $F$ of $\Gamma$, we identify $\Cb[F]$ with a subspace of $\ell^2(\Gamma)$ naturally.
In particular, if $F'\subseteq F$ are finite subsets of $\Gamma$, then $\Cb[F]$ is the direct sum of $\Cb[F']$ and $\Cb[F\setminus F']$.

\begin{lemma} \label{L-rational}
Let $f\in \Zb \Gamma$ be invertible in $\cL\Gamma$.
Then for any $\lambda>1$ and $C_1\ge 1$, there is some $\delta>0$ such that, for any $M\ge 1$ and any nonempty finite subsets
$F'\subseteq F$ of $\Gamma$ satisfying $|K_fF\setminus F|\le \delta|F|$ and $|F\setminus F'|\le \delta |F|$, if
$T_F$ is a linear map $\Cb[F\setminus F']\rightarrow \Cb[F]$
with $MT_F(\Zb[F\setminus F'])\subseteq \Zb[F]$ and $\|T_F\|\le C_1$ so that the linear map $S_F: \Cb[F]\rightarrow \Cb[F]$ defined
as $f_F$ on $\Cb[F']$ and $T_F$ on $\Cb[F\setminus F']$ is invertible in $B(\Cb[F])$,
then
$$ C\lambda^{|F|}M^{|K_fF\setminus F|}r_{F, \infty}(\frac{1}{8\|f\|_1})\ge |\det S_F|,$$
where $C$ is the universal constant in Lemma~\ref{ball:lemma}.
\end{lemma}
\begin{proof} The proof is similar to that of Lemma~\ref{L-positive large}. Write $K$ for $K_f$.
Set $D=8\|f\|_1$ and $\varepsilon=D^{-1}$. Take $1>\delta>0$ such that
$(2D(\|f\|+C_1)\|f^{-1}\|)^{2\delta}\le \lambda^{1/2}$, and $\delta^{1/2}\le \|f^{-1}\|$, and that $\delta'=2\delta$  satisfies the conclusion of
Lemma~\ref{ball:lemma} for $\lambda'=\lambda^{1/2}$.
Let $F$, $F'$,
and $T_F$ satisfy the hypothesis.

Consider $S'_F\in B(\Cb[F])$ defined as $f_F$ on $\Cb[F']$ and $MT_F$ on $\Cb[F\setminus F']$.
Then $\det S'_F=M^{|F\setminus F'|}\det S_F$ and $\|S'_F\|\le \|f\|+MC_1\le (\|f\|+C_1)M$.
Note that
$S'_F(\Zb[F])\subseteq \Zb[F]$, and hence $\det S'_F=|\Zb[F]/S'_F\Zb[F]|$ by Lemma~\ref{L-determinant quotient group size}.
Thus it suffices to
show
$$ C\lambda^{|F|}M^{|KF\setminus F'|}r_{F, \infty}(\varepsilon)\ge |\Zb[F]/S'_F\Zb[F]|.$$

Let $x\in \Zb[F]/S'_F\Zb[F]$. Take
$\tilde{x}\in \Zb[F]$ such that the image of $\tilde{x}$ in $\Zb[F]/S'_F\Zb[F]$ under the quotient
map $\Zb[F]\rightarrow \Zb[F]/S'_F\Zb[F]$ is
equal to $x$. Since $S'_F$ is invertible, one has
$$ \tilde{x}=S'_Fw$$
for some $w\in \Rb[F]$. Write
$w$ as $w_1+w_2$ for some $w_1\in \Zb[F]$ and $w_2\in
[0,1)^F$. Then $\tilde{x}=S'_Fw_1+S'_Fw_2$ and
$$\|S'_Fw_2\|_2\le \|S'_F\|\cdot \|w_2\|_2\le (\|f\|+C_1)M|F|^{1/2}.$$
Note that $\tilde{x}$ and
$S'_Fw_2$ have the same image in $\Zb[F]/S'_F\Zb[F]$. Thus we may replace
$\tilde{x}$ by $S'_Fw_2$ and hence assume that
$\|\tilde{x}\|_2\le (\|f\|+C_1)M|F|^{1/2}$.

Denote by $\varphi$ the quotient map $\Rb[[\Gamma]]\rightarrow
(\Rb/\Zb)[[\Gamma]]$.
For each $x\in \Zb[F]/S'_F\Zb[F]$, one has
$$ f\varphi(f^{-1}\tilde{x})=\varphi(f(f^{-1}\tilde{x}))=\varphi(\tilde{x})=0$$
in $(\Rb/\Zb)[[\Gamma]]$, and hence $\varphi(f^{-1}\tilde{x})\in
X_f$ by \eqref{E-X}.
This defines a map $\psi: \Zb[F]/S'_F\Zb[F]\rightarrow X_f$ sending $x$ to $\varphi(f^{-1}\tilde{x})$.

For each $x\in \Zb[F]/S'_F\Zb[F]$, pick $w_x\in
\frac{1}{D}\Zb[KF\setminus F]$ such that
$$\|w_x-p_{KF\setminus
F}(f^{-1}\tilde{x})\|_{\infty}\le 1/D=\varepsilon$$
and $|w_x(t)|\le |(f^{-1} \tilde{x})(t)|$ for all $t\in KF\setminus F$.
Then $Dw_x\in \Zb[KF\setminus F]$ and
$$\|Dw_x\|_2\le D\|p_{KF\setminus F}(f^{-1} \tilde{x})\|_2\le
D\cdot \|f^{-1}\|\cdot \|\tilde{x}\|_2 \le D(\|f\|+C_1)M\|f^{-1}\|\cdot
|F|^{1/2}.$$

Take an $[F, \infty, \varepsilon]$-spanning subset $E\subseteq X_f$ with $|E|=r_{F, \infty}(\varepsilon)$.
For each $v\in E$ set $W_v=\{x\in \Zb[F]/S'_F\Zb[F]: d_{F, \infty}(\psi(x), v)\le \varepsilon\}$.
Then $\bigcup_{v\in E}W_v=\Zb[F]/S'_F\Zb[F]$.
Now it suffices to show that
$$|W_v|\le C\lambda^{|F|}M^{|KF\setminus F'|}$$
for each $v\in E$.
Fix $v\in E$ and $y\in W_v$.

Let $x\in W_v$. Then
$$\max_{\gamma \in F}\vartheta((\psi(x))_\gamma, (\psi(y))_\gamma)=d_{F, \infty}(\psi(x), \psi(y))\le d_{F, \infty}(\psi(x), v)+d_{F, \infty}(\psi(y), v)\le 2\varepsilon.$$
For each $\gamma \in F'$, take $(h_x)_\gamma \in \Zb$ such that $|(f^{-1}\tilde{x})_\gamma-(f^{-1}\tilde{y})_\gamma-(h_x)_\gamma|\le 2\varepsilon$.
Similarly, for each $\gamma \in F\setminus F'$, take $(\theta_x)_\gamma \in \Zb$ such that $|(f^{-1}\tilde{x})_\gamma-(f^{-1}\tilde{y})_\gamma-(\theta_x)_\gamma|\le 2\varepsilon$.
Define $h_x\in \Zb[F']$ to be the element taking value $(h_x)_\gamma$ at each $\gamma \in F'$.
Also define $\theta_x\in \Zb[F\setminus F']$ to be the element taking value $(\theta_x)_\gamma$ at each $\gamma \in F\setminus F'$.
Set
\begin{eqnarray} \label{E-compare}
z_x=f^{-1}\tilde{x}-f^{-1}\tilde{y}-h_x-\theta_x-w_x+w_y\in \Rb[[\Gamma]].
\end{eqnarray}
Then
$$\|z_x|_{F'}\|_\infty=\|(f^{-1}\tilde{x}-f^{-1}\tilde{y}-h_x)|_{F'}\|_\infty\le 2\varepsilon,$$
and
$$ \|z_x|_{F\setminus F'}\|_\infty=\|(f^{-1}\tilde{x}-f^{-1}\tilde{y}-\theta_x)|_{F\setminus F'}\|_\infty\le 2\varepsilon,$$
and
\begin{align*}
\|z_x|_{KF\setminus F}\|_\infty&=\|(f^{-1}\tilde{x}-f^{-1}\tilde{y}-w_x+w_y)|_{KF\setminus F}\|_\infty\\
&\le \|(f^{-1}\tilde{x}-w_x)|_{KF\setminus F}\|_\infty+\|(f^{-1}\tilde{y}-w_y)|_{KF\setminus F}\|_\infty\le 2\varepsilon,
\end{align*}
and thus
$$ \|z_x|_{KF}\|_\infty\le 2\varepsilon.$$
It follows that
\begin{eqnarray*}
\| \theta_x\|_2&\le& \|f^{-1}\tilde{x}\|_2+\|f^{-1}\tilde{y}\|_2+\|p_{F\setminus F'}(z_x)\|_2\\
&\le& 2(\|f\|+C_1)M\|f^{-1}\|\cdot |F|^{1/2}+(2\varepsilon)\delta^{1/2}|F|^{1/2}\\
&\le& D(\|f\|+C_1)M\|f^{-1}\|\cdot |F|^{1/2}.
\end{eqnarray*}

Note that $\theta_x+Dw_x\in \Zb[KF\setminus F']$ with $\|\theta_x+Dw_x\|_2\le 2D(\|f\|+C_1)M\|f^{-1}\|\cdot |F|^{1/2}$
and $|KF\setminus F'|\le 2\delta|F|=\delta'|F|$.
By
Lemma~\ref{ball:lemma} one has
\begin{eqnarray*}
|\{\theta_x+Dw_x: x\in W_v\}|&\le &C\lambda^{|F|/2}(2D(\|f\|+C_1)M\|f^{-1}\|)^{|KF\setminus F'|}\\
&\le &C\lambda^{|F|/2}(2D(\|f\|+C_1)\|f^{-1}\|)^{2\delta|F|}M^{|KF\setminus F'|}\\
&\le & C\lambda^{|F|}M^{|KF\setminus F'|}.
\end{eqnarray*}
Thus
we can find a subset $W'_v\subseteq  W_v$ with $C\lambda^{|F|}M^{|KF\setminus F'|}|W'_v|\ge |W_v|$ such that
$\theta_{x_1}+Dw_{x_1}=\theta_{x_2}+Dw_{x_2}$ for all $x_1, x_2\in W'_v$. Since $\theta_x\in \Rb[F\setminus F']$
and $w_x\in \Rb[KF\setminus F]$ for all $x\in W'_v$, we have $\theta_{x_1}=\theta_{x_2}$ and $w_{x_1}=w_{x_2}$ for
$x_1, x_2\in W'_v$.

Now it suffices to show that $|W'_v|\le 1$.
Suppose that $x_1\neq x_2$ in $W'_v$.
Applying (\ref{E-compare}) to $x=x_1$ and $x=x_2$ respectively, one gets
$$ f^{-1}\widetilde{x_1}-f^{-1}\widetilde{x_2}=h_{x_1}-h_{x_2}+z_{x_1}-z_{x_2}.$$
Write $z_{x_1}-z_{x_2}$ as $z_1+z_2$ such that the supports of $z_1$ and $z_2$
are contained in $KF$ and $\Gamma\setminus KF$ respectively. Note
that $p_F(f(z_{x_1}-z_{x_2}))=p_F(fz_1)$ and
$\|z_1\|_{\infty}\le 4\varepsilon$. Consequently,
$$\|p_F(f(z_{x_1}-z_{x_2}))\|_{\infty}=\|p_F(f z_1)\|_{\infty}\le \| f z_1\|_{\infty}\le \|f\|_1\cdot \|z_1\|_{\infty}\le 4\varepsilon \|f\|_1=1/2.$$
We have
\begin{eqnarray*}
 \widetilde{x_1}-\widetilde{x_2}
 &=&p_F(\widetilde{x_1}-\widetilde{x_2})=p_F(f(h_{x_1}-h_{x_2}))+p_F(f(z_{x_1}-z_{x_2}))\\
 &=& S'_F(h_{x_1}-h_{x_2})+p_F(f(z_{x_1}-z_{x_2})).
\end{eqnarray*}
Since $\widetilde{x_1}-\widetilde{x_2}$ and $S'_F(h_{x_1}-h_{x_2})$ are both in $\Zb[F]$, we
must have $p_F(f(z_{x_1}-z_{x_2}))=0$. Therefore $\widetilde{x_1}-\widetilde{x_2}=S'_F(h_{x_1}-h_{x_2})\in
S'_F\Zb[F]$, contradicting the assumption $x_1\neq x_2$. This finishes the proof of
the lemma.
\end{proof}

\begin{remark} \label{R-uniform bound}
Note that in Lemma~\ref{L-rational} the operator $S_F$ may fail to preserve $\Zb[F]$, while the norm $\|S'_F\|$
of the operator $S'_F$ defined in the 2nd paragraph of the proof of Lemma~\ref{L-rational}
may be large when $M$ gets large.
Indeed, this is what happens when we construct $S_n$ in the proof of Lemma~\ref{L-lower bound} below.
A modification of  the proofs of Lemmas~\ref{L-positive small}, \ref{L-positive large} and \ref{L-rational} shows
that if $f\in \Zb\Gamma$ is invertible in $\cL\Gamma$, and there are a (left) F{\o}lner net $\{F_n\}_{n\in J}$ of $\Gamma$ and
an invertible $S_n\in B(\Cb[F_n])$  preserving $\Zb[F_n]$ for each $n\in J$
such that $\sup_{n\in J}\max(\|S_n\|, \|S_n^{-1}\|)<\infty$ and $\lim_{n\to \infty}\rank(S_n-f_{F_n})/|F_n|=0$, then
$\rh(\alpha_f)=\lim_{n\to \infty}\frac{1}{|F_n|}\log |\det S_n|$. Combined with Corollary~\ref{C-determinant formula}, this proves
Theorem~\ref{T-main} for such case, without using Theorem~\ref{T-positive} and Corollary~\ref{C-product}. When $\Gamma$ is also residually finite, Weiss showed that there are a net $\{\Gamma_n\}_{n\in J}$ of
finite index normal subgroups of $\Gamma$
and a (left) F{\o}lner net  $\{F_n\}_{n\in J}$ of $\Gamma$  such that the quotient map $\Gamma\rightarrow \Gamma/\Gamma_n$ maps $F_n$ bijectively
to $\Gamma/\Gamma_n$ for each $n\in J$ \cite[Section 2]{Weiss01} (see also \cite[Corollary 5.6]{DSch}). Via taking $S_n$ to be
 the image of $f$ in $\Cb(\Gamma/\Gamma_n)$ and identifying $B(\ell^2(\Gamma/\Gamma_n))$ and $B(\Cb[F_n])$, it is easily checked that
when $\Gamma$ is residually finite,  $\{F_n\}_{n\in J}$  and $\{S_n\}_{n\in J}$ satisfying the above conditions do exist. However, we have not been able to show the existence of such $\{F_n\}_{n\in J}$  and $\{S_n\}_{n\in J}$ in general. This is why we have to use Theorem~\ref{T-approximate determinant formula}, Lemma~\ref{L-lower bound} and Ornstein and Weiss's theory of quasitiling.
\end{remark}

For $\varepsilon>0$, we say that a family of finite subsets $\{F_1, \dots, F_m\}$ of $\Gamma$ are {\it $\varepsilon$-disjoint}
if there are $F'_j\subseteq F_j$ for all $1\le j\le m$ such that $F'_1, \dots, F'_m$ are pairwise disjoint, and
$|F'_j|\ge (1-\varepsilon)|F_j|$ for all $1\le j\le m$.
We need the following theorem of Ornstein and Weiss:

\begin{theorem}\cite[page 24, Theorem 6]{OW} \label{T-quasitile}
Let $\varepsilon>0$ and let $K$ be a nonempty finite subset of $\Gamma$.  Then there exist $\delta>0$ and
nonempty finite subsets $K', F_1, \dots, F_m$ of $\Gamma$,  such
that
\begin{enumerate}
\item $|\{g\in F_j: Kg\subseteq F_j\}|\ge (1-\varepsilon)|F_j|$ for each $1\le j\le m$;

\item for any nonempty finite subset $F$ of $\Gamma$ satisfying $|K'F\setminus F|\le \delta|F|$, there are
finite subsets $D_1, \dots,  D_m$ of $\Gamma$ such that $\bigcup_{1\le j\le m}F_jD_j\subseteq F$, the family
$\{F_jc: 1\le j\le m, c\in D_j\}$ of subsets of $\Gamma$ is $\varepsilon$-disjoint, and $|\bigcup_{1\le j\le m}F_jD_j|\ge (1-\varepsilon)|F|$.
\end{enumerate}
\end{theorem}

\begin{remark} \label{R-quasitile}
In Theorem~\ref{T-quasitile}, choosing $F_{c, j}\subseteq F_j$ for every $1\le j\le m$ and $c\in D_j$
such that $|F_{c, j}|\ge (1-\varepsilon)|F_j|$ for all $1\le j\le m$ and $c\in D_j$, and that
the family $\{F_{c,j}c: 1\le j\le m, c\in D_j\}$ of subsets of $\Gamma$ is pairwise disjoint,
and noticing that $F_{c, j}$ is one element in the finite set $\{W\subseteq F_j: |W|\ge (1-\varepsilon)|F_j|\}$, we see
that we can actually require the family $\{F_jc: 1\le j\le m, c\in D_j\}$ to be pairwise disjoint.
\end{remark}

\begin{lemma} \label{L-lower bound}
Let $\Gamma$ be an infinite amenable group and let $f\in \Zb\Gamma$ be invertible in $\cL\Gamma$. For any (left) F{\o}lner net $\{F_n\}_{n\in J}$ of $\Gamma$, one
has
$$\liminf_{n\to \infty}\frac{1}{|F_n|}\log r_{F_n, \infty}(\frac{1}{8\|f\|_1})\ge \log \det f.$$
\end{lemma}
\begin{proof} Set $C_1=\max(\|f\|, \|f^{-1}\|)+2$. Let $\lambda >1$ and $\varepsilon>0$.
Take $\delta>0$ working for both Theorem~\ref{T-approximate determinant formula} and Lemma~\ref{L-rational}.
Denote  by $K$ the union of the supports of $f$ and $f^*$, and the identity of $\Gamma$.

By Theorem~\ref{T-quasitile} and Remark~\ref{R-quasitile}, there exist
nonempty finite subsets $W_1, \dots, W_m$ of $\Gamma$ and $N\in J$,  such
that
\begin{enumerate}
\item[(I)] $|W'_j|\ge (1-\frac{\delta}{2})|W_j|$ for each $1\le j\le m$, where $W'_j=\{g\in W_j: Kg\subseteq W_j\}$;

\item[(II)] for any $n\ge N$, there are
finite subsets $D_{n,1}, \dots,  D_{n, m}$ of $\Gamma$ such that $\bigcup_{1\le j\le m}W_jD_{n, j}\subseteq F_n$, the family
$\{W_jc: 1\le j\le m, c\in D_{n,j}\}$ of subsets of $\Gamma$ is pairwise disjoint, and $|\bigcup_{1\le j\le m}W_jD_{n, j}|\ge (1-\frac{\delta}{2})|F_n|$.
\end{enumerate}
We may also assume that
\begin{enumerate}
\item[(III)] for any $n\ge N$, one has $|KF_n\setminus F_n|\le \delta |F_n|$.
\end{enumerate}

We shall construct $T_n$ for each $n\ge N$ satisfying the hypothesis in Lemma~\ref{L-rational}
for some $M$ not depending on $n$ such that the associated $S_n$ satisfies the hypothesis in Theorem~\ref{T-approximate determinant formula}.
For this purpose, we shall construct $T_n$ on $\Cb[W_j]$ first, then transfer them to $\Cb[F_n]$.

Fix $1\le j\le m$. Since $KW'_j\subseteq W_j$, we have $f_{W_j}=f$ on $\Cb[W'_j]$. Write $(f\Cb[W'_j])^{\perp}$ for the orthogonal complement of $f\Cb[W'_j]$ in $\Cb[W_j]$. Note that the dimension of $(f\Cb[W'_j])^{\perp}$ is equal to $|W_j\setminus W'_j|$, and
that $(f\Cb[W'_j])^{\perp}$ is the linear span of $(f\Cb[W'_j])^{\perp}\cap \Qb[W_j]$.
Identify $W_j$ with the standard orthonormal basis of $\Cb[W_j]$.
Take an orthonormal basis $\{e_g:g\in W_j\setminus W'_j\}$
of $(f\Cb[W'_j])^{\perp}$,  consisting of elements in $\Rb[W_j]$. Taking $e'_g\in (f\Cb[W'_j])^{\perp}\cap \Qb[W_j]$ close enough to
$e_g$ for all $g\in W_j\setminus W'_j$, we find that the linear map $\widetilde{T}_j: \Cb[W_j\setminus W'_j]\rightarrow (f\Cb[W'_j])^{\perp}$ sending $g$ to $e'_g$ is bijective and $\max(\|\widetilde{T}_j\|, \|\widetilde{T}_j^{-1}\|)\le 2$. Then there exists $M_j\in \Nb$ such that $M_j\widetilde{T}_j(\Zb[W_j\setminus W'_j])\subseteq \Zb[W_j]$. Note that the linear map $\widetilde{S}_j: \Cb[W_j]\rightarrow \Cb[W_j]$
defined as $f_{W_j}$ on $\Cb[W'_j]$ and $\widetilde{T}_j$ on $\Cb[W_j\setminus W'_j]$ is invertible,
and $\|\widetilde{S}_j^{-1}\|\le \|f^{-1}\|+2$.

Set $M=\prod_{1\le j\le m}M_j$.

Now let $n\ge N$. Let $D_{n, 1}, \dots, D_{n, m}$ be as in (II) above. Set $F'_n=\bigcup_{1\le j\le m}W'_jD_{n, j}$. Then $|F_n\setminus F'_n|\le \delta|F_n|$. Next we define the desired  linear map $T_n: \Cb[F_n\setminus F'_n]\rightarrow \Cb[F_n]$. On $\Cb[F_n\setminus (\bigcup_{1\le j\le m}W_jD_{n, j})]$, the map $T_n$ is the identity map. On $\Cb[(W_j\setminus W'_j)c]$ for $1\le j\le m$ and $c\in D_{n, j}$, the map $T_n$ is the same
as $\widetilde{T}_j$ on $\Cb[W_j\setminus W'_j]$, if we identify $\Cb[W_j\setminus W'_j]$ and $\Cb[W_j]$ with $\Cb[(W_j\setminus W'_j)c]$
and $\Cb[W_jc]$ respectively via the right multiplication by $c$. Then $MT_n(\Zb[F_n\setminus F'_n])\subseteq \Zb[F_n]$, and $\|T_n\|\le 2$.
Denote by $S_n$ the linear map
$\Cb[F_n]\rightarrow \Cb[F_n]$ which is equal to $f_{F_n}$ on $\Cb[F'_n]$ and equal to $T_n$ on $\Cb[F_n\setminus F'_n]$.
Clearly $\|S_n\|\le \|f\|+2$.
 Note that the restriction of $S_n$ on $\Cb[W_jc]$ for each $1\le j\le m$ and $c\in D_{n, j}$, or on $\Cb[F_n\setminus  (\bigcup_{1\le j\le m}W_jD_{n, j})]$ is an isomorphism, and the norm of the inverse of this restriction is bounded above by $\|f^{-1}\|+2$. Thus $S_n$ is invertible with $\|S_n^{-1}\|\le \|f^{-1}\|+2$. By Lemma~\ref{L-rational} we have
\begin{eqnarray*}
C\lambda^{|F_n|}M^{|KF_n\setminus F_n|}r_{F_n, \infty}(\frac{1}{8\|f\|_1})\ge |\det S_n|,
\end{eqnarray*}
where $C$ is the universal constant in Lemma~\ref{ball:lemma}. Therefore
\begin{eqnarray} \label{E-lower bound}
\liminf_{n\to \infty}(\frac{1}{|F_n|}\log r_{F_n, \infty}(\frac{1}{8\|f\|_1})-\frac{1}{|F_n|}\log |\det S_n|)\ge -\log \lambda.
\end{eqnarray}

 Since $S_n$ and $f_{F_n}$ coincide on $\Cb[F'_n]$, we have $\rank(S_n-f_{F_n})\le |F_n\setminus F'_n|\le \delta |F_n|$.
By Theorem~\ref{T-approximate determinant formula} we get
\begin{eqnarray} \label{E-approximate}
\limsup_{n\to \infty}\big|\log \ddet_{\cL\Gamma} f-\frac{1}{|F_n|}\log |\det S_n|\big|<\varepsilon.
\end{eqnarray}
Combining (\ref{E-lower bound}) and (\ref{E-approximate}) we get
$$ \liminf_{n\to \infty}(\frac{1}{|F_n|}\log r_{F_n, \infty}(\frac{1}{8\|f\|_1})-\log \ddet_{\cL\Gamma} f)\ge -\log \lambda-\varepsilon.$$
Since $\lambda>1$ and $\varepsilon>0$ are arbitrary, the lemma is proved.
\end{proof}

\section{Proof of Theorem~\ref{T-main} and Consequences} \label{S-consequence}

We are ready to prove Theorem~\ref{T-main}.

\begin{proof}[Proof of Theorem~\ref{T-main}]
By Theorem~\ref{T-finite} we may assume that $\Gamma$ is infinite.
Let $\{F_n\}_{n\in \Nb}$ be a (left) F{\o}lner sequence of $\Gamma$.
By Theorem~\ref{T-sep} and Lemma~\ref{L-lower bound} we have
$$ \rh(\alpha_f)\ge \liminf_{n\to \infty}\frac{1}{|F_n|}\log r_{F_n, \infty}(\frac{1}{8\|f\|_1})\ge \log \ddet_{\cL\Gamma} f.$$
Applying the inequality to $f^*$, we also have
$$ \rh(\alpha_{f^*})\ge \log \ddet_{\cL\Gamma} f^*.$$

Then we have
\begin{eqnarray*}
 \rh(\alpha_{f^*f})&=&\rh(\alpha_{f^*})+\rh(\alpha_f)\ge \log \ddet_{\cL\Gamma} f^*+\log \ddet_{\cL\Gamma} f=2\log \ddet_{\cL\Gamma} f\\
 &=&\log \ddet_{\cL\Gamma} (f^*f)=\rh(\alpha_{f^*f}),
\end{eqnarray*}
where the first equality comes from Corollary~\ref{C-product}, the second one comes from Theorem~\ref{T-FK}, the third one comes from the definition of $\ddet_{\cL\Gamma}f$,
 and the last one comes from Theorem~\ref{T-positive}.
Thus $ \rh(\alpha_f)=\log \ddet_{\cL\Gamma} f$.
\end{proof}

Since $s_{F, \infty}(\varepsilon)\ge r_{F, \infty}(\varepsilon)$ for any nonempty finite subset $F$ of $\Gamma$ and $\varepsilon>0$,
in the proof of Theorem~\ref{T-main} we actually have proved the following result.

\begin{corollary} \label{C-entropy formula}
Let $\Gamma$ be a countable amenable group and let $f\in \Zb\Gamma$ be invertible in $\cL\Gamma$. For any $\frac{1}{8\|f\|_1}\ge \varepsilon>0$ and any (left) F{\o}lner sequence $\{F_n\}_{n\in \Nb}$ of $\Gamma$, one
has
$$\rh(\alpha_f)=\lim_{n\to \infty}\frac{1}{|F_n|}\log r_{F_n, \infty}(\varepsilon)=\lim_{n\to \infty}\frac{1}{|F_n|}\log s_{F_n, \infty}(\varepsilon).$$
\end{corollary}

The follow result is a consequence of Theorems~\ref{T-main} and \ref{T-FK}.

\begin{corollary} \label{C-adjoint}
Let $\Gamma$ be a countable amenable group and let  $f\in \Zb\Gamma$ be invertible in $\cL\Gamma$. Then $\rh(\alpha_f)=\rh(\alpha_{f^*})$.
\end{corollary}

The following result is a generalization of \cite[Corollary 6.6]{DSch}, whose proof we follow.

\begin{corollary} \label{C-monotone}
Let $\Gamma$ be a countable amenable group and let $f, g\in \Zb\Gamma$ be invertible in $\cL\Gamma$ with $0\le f\le g$. Then $\rh(\alpha_f)\le \rh(\alpha_g)$. Furthermore, $ \rh(\alpha_f)=\rh(\alpha_g)$
if and only if $f=g$.
\end{corollary}
\begin{proof} The inequality follows from Theorems~\ref{T-main} and \ref{T-FK}. Suppose that $\rh(\alpha_f)=\rh(\alpha_g)$.
Set $h=\log g-\log f$. Then
$$\tr_{\cL\Gamma}(h)=\tr_{\cL\Gamma} \log g-\tr_{\cL\Gamma} \log f=\log \ddet_{\cL\Gamma}g-\log \ddet_{\cL\Gamma}f=\rh(\alpha_g)-\rh(\alpha_f)=0.$$
Note that the function
$\log$ is {\it operator monotone}
in the sense that for any invertible bounded positive operators $T, S$ on a Hilbert space $H$ with $T\le S$,  one has $\log T\le \log S$ \cite{Lowner} \cite[Page 10]{Ped}. Thus $h\ge 0$. Since $\tr_{\cL\Gamma}$ is faithful and $\tr_{\cL\Gamma}h=0$, we get $h=0$. Thus $f=g$.
\end{proof}

\appendix

\section{Comparison of Invertibility in $\ell^1(\Gamma)$ and $\cL\Gamma$}  \label{A-invertibility}

A Banach complex algebra $A$ with an operation $*$ is called a {\it Banach $*$-algebra} if
$(a^*)^*=a$, $(\lambda a+b)^*=\bar{\lambda}a^*+b^*$, $(ab)^*=b^*a^*$, and $\|a^*\|=\|a\|$ for all
$a, b\in A$ and $\lambda \in \Cb$. A {\it representation} of a Banach $*$-algebra $A$ on a Hilbert
space $H$ is a $*$-homomorphism $\pi:A\rightarrow B(H)$. A Banach $*$-algebra $A$ is called
an $A^*$-algebra if it has an injective representation $\pi$. For an $A^*$-algebra $A$,
there is a $C^*$-algebra $C^*(A)$ and an injective $*$-homomorphism $A\hookrightarrow C^*(A)$
with dense image such that every $*$-homomorphism $A\rightarrow \sB$ of $A$ into a $C^*$-algebra $\sB$
extends to a unique $*$-homomorphism $C^*(A)\rightarrow \sB$. The $C^*$-algebra $C^*(A)$ is unique up to isomorphism,
and is called the {\it enveloping $C^*$-algebra} of $A$ \cite[page 42]{Takesaki}.
Explicitly, the norm $\|\cdot \|_{C^*(A)}$ of $C^*(A)$ is given by
$\|a\|_{C^*(A)}=\sup_{\pi}\|\pi(a)\|$ for $a\in A$, where $\pi$ runs over all representations
of $A$.

A unital Banach $*$-algebra $A$ is called {\it symmetric} if
for each $a\in A$, the spectrum of $a^*a$ in $A$ is contained in $\Rb_{\ge 0}$. It is well known that
a unital  $A^*$-algebra $A$ is symmetric if and only if for each $a\in A$, the spectra of $a$ in $A$ and
$C^*(A)$ are the same.
We recall briefly the reason here. The ``if" part follows from the fact that every $C^*$-algebra is
symmetric. Assume that $A$ is symmetric. By a result of Raikov \cite{Raikov} \cite[page 308, Corollary 4]{Naimark},
for every $a\in A$ with $a^*=a$, the spectral radius of $a$ in $A$ is equal to $\|a\|_{C^*(A)}$. According to
an observation of Hulanicki \cite[Proposition 2.5]{Hulanicki72} (see also \cite[Proposition 6.1]{FGLLM}),
for every $a\in A$ with $a^*=a$, the spectra of $a$ in $A$ and $C^*(A)$ are the same. It follows that
for every $a\in A$, the spectra of $a$ in $A$ and $C^*(A)$ are the same (see for example \cite[page 804]{FGLLM}).

Let $\Gamma$ be a discrete (not necessarily amenable) group. Then $\ell^1(\Gamma)$ is a unital Banach $*$-algebra with
the algebraic operations extending those of $\Cb\Gamma$. The embedding $\Cb\Gamma\hookrightarrow \cL\Gamma$ extends to an injective
representation $\ell^1(\Gamma)\hookrightarrow \cL\Gamma$.  Thus $\ell^1(\Gamma)$ is an $A^*$-algebra, and for
every $a\in \ell^1(\Gamma)$, if $a$ is invertible in $\ell^1(\Gamma)$, then it is invertible in $\cL\Gamma$.  The enveloping
$C^*$-algebra of $\ell^1(\Gamma)$ is denoted by $C^*(\Gamma)$ and is called the {\it (maximal) group $C^*$-algebra} of $\Gamma$. The embedding
$\ell^1(\Gamma)\hookrightarrow \cL\Gamma$ extends to a $*$-homomorphism $\psi: C^*(\Gamma)\rightarrow \cL\Gamma$.
The group $\Gamma$ is amenable if and only if $\psi$ is injective \cite[Theorem 4.21]{Pat}.
Thus, when $\Gamma$ is amenable, $\ell^1(\Gamma)$ is symmetric if and only if for any $a\in \ell^1(\Gamma)$, the spectra of $a$ in $\ell^1(\Gamma)$
and $\cL\Gamma$ are the same.

If $\Gamma$ is a finite extensions of a discrete nilpotent group, then $\ell^1(\Gamma)$ is symmetric \cite{LP, Ludwig}, and hence
for any $a\in \ell^1(\Gamma)$, $a$ is invertible in $\ell^1(\Gamma)$ if and only if it is invertible in $\cL\Gamma$.

Nica showed that if $\Gamma$ is a finitely generated group of subexponential growth, then for any $a\in \Cb\Gamma$, $a$ is invertible in $\ell^1(\Gamma)$ if and only if it is invertible in $\cL\Gamma$ \cite[page 3309]{Nica}. 

Jenkins \cite{Jenkins70, Jenkins71} showed that if $\Gamma$ is a discrete group containing two elements generating
a free subsemigroup, then $\ell^1(\Gamma)$ is not symmetric. 
Under the same assumption, Nica showed that there exist $a\in \Cb\Gamma$ which are invertible in $\cL\Gamma$ but not invertible in $\ell^1(\Gamma)$ \cite[Proposition 52]{Nica}.
In fact, in such case there exist $a\in \Zb\Gamma$
which are invertible in $C^*(\Gamma)$ (in particular, invertible in $\cL\Gamma$) but not invertible in $\ell^1(\Gamma)$, as the following example shows.
This example is inspired by the ideas in \cite{Jenkins71}.
I am grateful to Jingbo Xia for very helpful discussion leading to this example.

\begin{example} \label{E-not invertible}
Let $\Gamma$ be a discrete group with elements $\gamma_1, \gamma_2\in \Gamma$ generating a free subsemigroup.
We claim that for every $\lambda \in \Cb$ with $|\lambda|=3$, the element
$$a=\lambda e_{\Gamma}-(e_{\Gamma}+\gamma_1-\gamma_1^2)\gamma_2$$
is invertible in $C^*(\Gamma)$ but
not invertible in $\ell^1(\Gamma)$. Taking $\lambda=\pm 3$, we get $a\in \Zb\Gamma$.
The spectrum  of $\gamma_1$ in $C^*(\Gamma)$ is contained in the unit circle $\Tb$ of $\Cb$. By the spectral theorem
for unitaries,
$$ \|e_{\Gamma}+\gamma_1-\gamma_1^2\|_{C^*(\Gamma)} \le \max_{z\in \Tb}|1+z-z^2|<3.$$
Then
\begin{eqnarray*}
\|(e_{\Gamma}+\gamma_1-\gamma_1^2)\gamma_2\|_{C^*(\Gamma)}&\le &\|e_{\Gamma}+\gamma_1-\gamma_1^2\|_{C^*(\Gamma)}\cdot \|\gamma_2\|_{C^*(\Gamma)}\\
&=&\|e_{\Gamma}+\gamma_1-\gamma_1^2\|_{C^*(\Gamma)}<3.
\end{eqnarray*}
It follows that $a$ is invertible in $C^*(\Gamma)$, and its inverse is given by
$$ \lambda^{-1}\sum_{k=0}^{\infty}\lambda^{-k}( (e_{\Gamma}+\gamma_1-\gamma_1^2)\gamma_2)^k.$$
From the natural homomorphism $C^*(\Gamma)\rightarrow \cL\Gamma$, we see that $a$ is also invertible in
$\cL\Gamma$ with inverse $b$ given by the above formula. Under the natural embedding $\cL\Gamma\rightarrow \ell^2(\Gamma)$,
$b=\lambda^{-1}\sum_{k=0}^{\infty}\lambda^{-k}( (e_{\Gamma}+\gamma_1-\gamma_1^2)\gamma_2)^k\in \ell^2(\Gamma)$.
Since $\gamma_1$ and $\gamma_2$ generate a free subsemigroup, it is easily checked that
the supports of $( (e_{\Gamma}+\gamma_1-\gamma_1^2)\gamma_2)^k$ for $k\ge 0$ are pairwise disjoint and
$\|( (e_{\Gamma}+\gamma_1-\gamma_1^2)\gamma_2)^k\|_1=3^k$ for each $k\ge 0$. It follows that
$\lambda^{-1}\sum_{k=0}^{\infty}\lambda^{-k}( (e_{\Gamma}+\gamma_1-\gamma_1^2)\gamma_2)^k\not \in \ell^1(\Gamma)$. If $a$ were invertible in
$\ell^1(\Gamma)$, then its inverse in $\ell^1(\Gamma)$ would be $b$ and hence $b\in \ell^1(\Gamma)$, which is a contradiction.
Therefore $a$ is not invertible in $\ell^1(\Gamma)$.
\end{example}

There are discrete amenable groups which contain two elements generating free subsemigroups \cite{Hochster, Jenkins69}.
Actually Frey showed that every discrete amenable group with nonamenable subsemigroups has such elements \cite{Frey}.
Also, Chou showed that if a finitely generated elementary amenable group has no finite-index nilpotent subgroups, then it contains such elements \cite[Theorem 3.2']{Chou}.
We recall the examples in \cite{Jenkins69}. Consider the action of the multiplicative group $\Rb^*=\Rb\setminus \{0\}$
on the additive group $\Rb$ by multiplication. One has the semi-direct product group $\Rb \rtimes\Rb^*$, which
is $\Rb\times \Rb^*$ as a set and has multiplication $(s_1, t_1)\cdot (s_2, t_2)=(s_1+t_1s_2, t_1t_2)$. For
any $0\le a\le 1/2$, the subgroup $\Gamma_a$ of $\Rb \rtimes\Rb^*$ generated by $(1, a)$ and $(1, -a)$ is $2$-step solvable (and hence amenable),
and $(1, a)$ and $(1, -a)$ generate a free subsemigroup in $\Gamma_a$.


\end{document}